\documentclass[9.8pt,letterpaper]{amsart}
\usepackage[utf8]{inputenc}
\usepackage{amsmath}
\usepackage{amsfonts}
\usepackage{amssymb}
\usepackage{amsthm}
\usepackage{chngcntr}
\usepackage{mathtools}
\usepackage[all,cmtip]{xy}
\usepackage{tikz}
\usepackage[pagebackref=false,colorlinks]{hyperref}
\hypersetup{pdffitwindow=true,linkcolor=blue,citecolor=blue,urlcolor=cyan}
\usepackage{cite}
\usepackage{fancyhdr}
\usepackage{stmaryrd}
\usepackage{yfonts}

\font\wncyr=wncyr9.8
\newcommand{\sha}{\text{\wncyr{W}}}

\usepackage[top=2.6cm, bottom=2.4cm, left=2.4cm, right=2.4cm]{geometry}

\usepackage{tikz}
\usepackage{tikz-cd}

\usepackage{hyperref}

\newtheorem{theorem}{Theorem}
\newtheorem{lemma}[theorem]{Lemma}
\newtheorem{prop}[theorem]{Proposition}
\newtheorem{cor}[theorem]{Corollary}

\newtheorem{notation}[theorem]{Notation}
\newtheorem{ques}[theorem]{Question}

\theoremstyle{definition}
\newtheorem{definition}[theorem]{Definition}
\newtheorem{remark}[theorem]{Remark}

\newtheorem{theoremintro}{Theorem}

\renewcommand{\thetheorem}{\arabic{section}.\arabic{subsection}.\arabic{theorem}}

\makeatletter
\@addtoreset{theorem}{section}
\makeatother

\begin{document}

\title{Structure of the Mordell-Weil group over the $\mathbb{Z}_p$-extensions}

\author{JAEHOON LEE}
\maketitle

\begin{abstract} We study the $\Lambda$-module structure of the Mordell-Weil, Selmer, and Tate-Shafarevich groups of an abelian variety over $\mathbb{Z}_p$-extensions.
\end{abstract}

\tableofcontents

\smallskip

\section{Introduction}

\subsection{Overview and the questions}

The Iwasawa theory of abelian varieties (elliptic curves) was initiated by Mazur in his seminal paper \cite{mazur1972rational}, where he proved that the $p^{\infty}$-Selmer groups of an abelian variety $A$ defined over a number field $K$ are ``well-controlled" over $\mathbb{Z}_p$-extensions if $A$ satisfies certain reduction condition at places dividing $p$. More precisely, let $K_{\infty}$ be a $\mathbb{Z}_p$-extension of $K$ and $K_n$ be the $n$-th layer, and $\mathrm{Sel}_{K_n}(A)_p$ be the classical $p^{\infty}$-Selmer group over $K_n$ attached to $A$. For a natural restriction map
$$S_{n}^{A}:\mathrm{Sel}_{K_n}(A)_p \rightarrow \mathrm{Sel}_{K_{\infty}}(A)_p[\omega_n],$$ we have the following celebrated theorem.
\\

\begin{theorem}[Control Theorem] 
If either 
\begin{itemize}
\item $(\mathrm{Mazur}$, \cite{mazur1972rational}$)$ $A$ has a good ordinary reduction at all places of $K$ dividing $p$
\end{itemize}
or
\begin{itemize}
\item $(\mathrm{Greenberg}$, \cite[Proposition 3.7]{greenberg1999iwasawa}$)$ $K=\mathbb{Q}$ and $A$ is an elliptic curve having multiplicative reduction at $p$
\end{itemize}

then $\mathrm{Coker}(S_{n}^{A})$ is finite and bounded independent of $n$.
\end{theorem}

Here the word ``control" means that the Selmer group $\mathrm{Sel}_{K_n}(A)_p$ at each layer $K_n$ can be described by the \emph{one} object. Hence we can expect that the each Selmer group over $K_n$ should behave in a certain regular way governed by the ``limit" Selmer group $\mathrm{Sel}_{K_{\infty}}(A)_p$. For instance, we have the following consequence of the above theorem.\\

\begin{prop} Assume either one of the condition of Theorem 1.1.1 and also assume that both $A(K_n)$ and $\sha_{K_n}^{1}(A)_p$ are finite for all $n$. Then there exists $\mu, \lambda, \nu$ such that $$| \mathrm{Sel}_{K_n}(A)_p |=|\sha_{K_n}^{1}(A)_p|=p^{e_n} \quad (n>>0)$$ where $$e_n=p^n \mu + n \lambda + \nu.$$
\end{prop}

For the proof, see \cite[Corollary 4.11]{greenberg2001introduction}. The value $p^n \mu +n \lambda +\nu$ in the Proposition 1.1.2 naturally appears from the structure theory of $\Lambda$-modules. For a finitely generated $\Lambda$-module $M$, there is a $\Lambda$-linear map $$\displaystyle M \rightarrow \Lambda^{r}\oplus \left(\bigoplus_{i=1}^{n}\frac{\Lambda}{g_{i}^{e_i}}\right)\oplus \left( \bigoplus_{j=1}^{m}\frac{\Lambda}{p^{f_j}} \right)$$ with finite kernel and cokernel where $r, n, m \geq 0$, $e_1, \cdots, e_n, f_1, \cdots, f_m$ are positive integers, and $g_1, \cdots, g_n$ are distinguished irreducible polynomial of $\Lambda$. The quantities $r, e_1, \cdots , e_n, f_1, \cdots , f_m, g_1, \cdots , g_n$ are uniquely determined, and we call $$\displaystyle E(M):=\Lambda^{r}\oplus \left(\bigoplus_{i=1}^{n}\frac{\Lambda}{g_{i}^{e_i}}\right)\oplus \left( \bigoplus_{j=1}^{m}\frac{\Lambda}{p^{f_j}} \right)$$ as an \emph{elementary module of $M$} following \cite[Page 292]{neukirch2000cohomology}. The $\lambda$-invariant $\lambda(M)$ is defined as $\displaystyle \sum_{i=1}^{n}e_i \cdot \mathrm{deg}g_i$ and the $\mu$-invariant $\mu(M)$ is defined as $f_1+...+f_m$. \\

In Proposition 1.1.2, the constants $\lambda$ and $\mu$ are indeed the $\lambda$ and $\mu$-invariants of the module $\mathrm{Sel}_{K_{\infty}}(A)_p^{\vee}$, respectively and hence we can say that the module $E\left(  \mathrm{Sel}_{K_{\infty}}(A)_p^{\vee} \right)$ gives the information about the arithmetic of the $A$ at each finite level at once. Hence it is quite natural to ask questions about the \emph{shape} of the $\Lambda$-module $E\left(\mathrm{Sel}_{K_{\infty}}(A)_p^{\vee}\right)$. We can also consider the same question for the Mordell-Weil group and the Tate-Shafarevich group, which fit into a natural exact sequence $$0 \rightarrow A(K_{\infty})\otimes \mathbb{Q}_p/\mathbb{Z}_p \rightarrow \mathrm{Sel}_{K_{\infty}}(A)_p \rightarrow \sha_{K_{\infty}}^{1}(A)_p \rightarrow 0.$$ 

\medskip

The main goal of this paper is studying $$E\left( (A(K_{\infty}) \otimes_{\mathbb{Z}_p} \mathbb{Q}_p/\mathbb{Z}_p)^{\vee} \right), \quad E\left( \mathrm{Sel}_{K_{\infty}}(A)_p^{\vee} \right), \quad E\left( \sha_{K_{\infty}}^{1}(A)_p^{\vee}\right).$$ More precisely, we can ask the following questions:

\medskip

\begin{ques}[Describing the elementary modules]  \mbox{ }
\begin{itemize}
\item Can we describe the structure of $E\left( (A(K_{\infty}) \otimes_{\mathbb{Z}_p} \mathbb{Q}_p/\mathbb{Z}_p)^{\vee} \right), E\left( \mathrm{Sel}_{K_{\infty}}(A)_p^{\vee} \right), E\left( \sha_{K_{\infty}}^{1}(A)_p^{\vee}\right)?$ Or can we find some relations among these three modules? 
\end{itemize}
\end{ques}

\smallskip

\begin{ques}[``Smallness" of the $\sha^1$] Under the finiteness assumption of the groups $\sha_{K_n}^{1}(A)_p$, 
\begin{itemize}
\item Is the group $\sha_{K_{\infty}}^{1}(A)_p$ $\Lambda$-cotorsion?
\smallskip
\item Can we find an estimate (even conjecturally) of $|\sha_{K_n}^{1}(A)_p|$ in terms of $n$?
\end{itemize}
\end{ques}

\smallskip

\begin{ques}[Algebraic Functional Equation] If $A^t$ is the dual abelian variety of $A$, then can we find any relation between \begin{itemize}
\item $E\left( \mathrm{Sel}_{K_{\infty}}(A)_p^{\vee} \right) \quad \mathrm{and} \quad E\left( \mathrm{Sel}_{K_{\infty}}(A^t)_p^{\vee} \right)?$
\item $E\left( \sha_{K_{\infty}}^{1}(A)_p^{\vee}\right) \quad \mathrm{and} \quad E\left(\sha_{K_{\infty}}^{1}(A^t)_p^{\vee}\right)?$
\item $E\left( (A(K_{\infty}) \otimes_{\mathbb{Z}_p} \mathbb{Q}_p/\mathbb{Z}_p)^{\vee} \right) \quad \mathrm{and} \quad E\left( (A^t(K_{\infty}) \otimes_{\mathbb{Z}_p} \mathbb{Q}_p/\mathbb{Z}_p)^{\vee}\right)?$
\end{itemize}

\end{ques}

\smallskip

We will try to answer these questions in this paper. For the $\Lambda$-corank of the module $\mathrm{Sel}_{K_{\infty}}(A)_p$, we have the following famous theorem.

\smallskip

\begin{theorem}[Kato-Rohrlich] If $A$ is an elliptic curve defined over $\mathbb{Q}$ with good ordinary reduction or multiplicative reduction at $p$, $K$ is an abelian extension of $\mathbb{Q}$ and $K_{\infty}$ is a cyclotomic $\mathbb{Z}_p$-extension of $K$, then the $\mathrm{Sel}_{K_{\infty}}(A)_p$ is a cotorsion $\Lambda$-module.
\end{theorem}

\medskip

\begin{remark} The cotorsionness of $\mathrm{Sel}_{K_{\infty}}(A)_p$ is known \emph{only} (at this moment) under the assumptions of Theorem 1.1.6. 
\end{remark}

\medskip

\begin{remark}
The \emph{main novelty} of this paper are two folds:

\begin{itemize}
\item First of all, instead of the characteristic ideals of the modules above (which are usually studied because of its connection with the Iwasawa Main Conjecture), we study their $\Lambda$-module \emph{structure}. Here the word ``structure" means that we study the elementary modules of $$(A(K_{\infty}) \otimes_{\mathbb{Z}_p} \mathbb{Q}_p/\mathbb{Z}_p)^{\vee} , \mathrm{Sel}_{K_{\infty}}(A)_p^{\vee}, \sha_{K_{\infty}}^{1}(A)_p^{\vee}.$$ 
\medskip
\item Another point is that we do \emph{not} assume the \emph{$\Lambda$-cotorsionness} of the Selmer group. We solely assume the control of the Selmer group and study the three $\Lambda$-modules $(A(K_{\infty}) \otimes_{\mathbb{Z}_p} \mathbb{Q}_p/\mathbb{Z}_p)^{\vee} , \mathrm{Sel}_{K_{\infty}}(A)_p^{\vee}, \sha_{K_{\infty}}^{1}(A)_p^{\vee}.$ Moreover, our technique using the functors $\mathfrak{F}$ and $\mathfrak{G}$ (which will be introduced in Appendix) gives another proof of the known algebraic functional equation results. See 1.3 for the comparison with former works.
\end{itemize}
\end{remark}

\smallskip

We hope that our results enable us to determine the structure of the various arithmetic groups (e.g. Mordell-Weil, Selmer and $\sha^{1}$) as modules over the group ring at each finite layer $K_n$.

\medskip

\subsection{Main theorems and consequences}

\medskip

\begin{theoremintro}[Theorem 2.1.2] We have a $\Lambda$-linear injection $$\displaystyle (A(K_{\infty}) \otimes_{\mathbb{Z}_p} \mathbb{Q}_p/\mathbb{Z}_p)^{\vee} \hookrightarrow \Lambda^{r} 
 \oplus \left( \bigoplus_{n=1}^{t} \frac{\Lambda}{\omega_{b_{n}+1, b_n}} \right) $$ with finite cokernel for some integers $r, b_1, \cdots b_n$ where $\omega_{n+1, n}:=\frac{(1+T)^{p^{n+1}}-1}{(1+T)^{p^n}-1}$.
\end{theoremintro}

\medskip

\begin{remark} (1) Hence the direct factors of $E\left((A(K_{\infty}) \otimes_{\mathbb{Z}_p} \mathbb{Q}_p/\mathbb{Z}_p)^{\vee}\right)_{\Lambda-tor}$ are only of the form $\frac{\Lambda}{\omega_{k+1, k}}$ for some $k$. For instance, neither $\frac{\Lambda}{T^5+p}$ nor $\frac{\Lambda}{T^2}$ can not be a direct factor of $E\left((A(K_{\infty}) \otimes_{\mathbb{Z}_p} \mathbb{Q}_p/\mathbb{Z}_p)^{\vee}\right)_{\Lambda-tor}$.
\smallskip
(2) The same proof also works if $K$ is a finite extension of $\mathbb{Q}_p$. (See Page 7 for the proof.) 
\end{remark}

\medskip

One consequence of this Theorem in $p$-adic local case is the following:

\medskip

\begin{cor}[Theorem 2.2.7] Let $L$ be a finite extension of $\mathbb{Q}_p$. If $A_{/ L}$ has potentially supersingular reduction and $L_{\infty}/L$ is a ramified $\mathbb{Z}_p$-extension, then $A(L_{\infty})[p^{\infty}]$ is finite.\end{cor}

\medskip

\medskip

To analyze the $\Lambda$-module structure of the Tate-Shafarevich group, we introduce a functor $\mathfrak{G}$. For a finitely generated $\Lambda$-module $X$, we define $\displaystyle \mathfrak{G}(X):=\lim_{\substack {\longleftarrow \\ n}}\left( \frac{X}{\omega_nX}[p^{\infty}] \right)$. See the appendix-Proposition A.2.12 for the explicit description of this functor. In particular, $\mathfrak{G}(X)$ is a finitely generated torsion $\Lambda$-module. 

\smallskip

We have the following theorem which provides an answer for Question 1.1.3 and Question 1.1.4.

\medskip

\begin{theoremintro}[Theorem 3.0.4] If $\mathrm{Coker}(S_{n}^{A})$ and $\sha_{K_n}^{1}(A)_p$ are finite for all $n$, then we have an isomorphism $$\sha_{K_{\infty}}^{1}(A)_p^{\vee} \simeq \mathfrak{G}\left(\mathrm{Sel}_{K_{\infty}}(A)_p^{\vee}\right)$$ of $\Lambda$-modules. In particular, $\sha_{K_{\infty}}^{1}(A)_p$ is a cotorsion $\Lambda$-module. 
\end{theoremintro}

\medskip

\begin{remark}
(1) This theorem can be regarded as a $\Lambda$-adic analogue of the Tate-Shafarevich conjecture. More precisely, if the control theorem (of the Selmer groups) holds, the $\mathbb{Z}_p$-cotorsionness of $\sha_{K_n}^{1}(A)_p$ (i.e. finiteness) at finite level can be lifted to the $\Lambda$-cotorsionness of the $\sha_{K_{\infty}}^{1}(A)_p$.

\smallskip

(2) This theorem also describes the $\Lambda$-module structure of $\sha_{K_{\infty}}^{1}(A)_p^{\vee}$ by the structural data of the $\mathrm{Sel}_{K_{\infty}}(A)_p^{\vee}$: If we know the elementary module of $\mathrm{Sel}_{K_{\infty}}(A)_p^{\vee}$, then we can explicitly write the elementary module of the $\sha_{K_{\infty}}^{1}(A)_p^{\vee}$ and $\left(A(K_{\infty})\otimes \mathbb{Q}_p/\mathbb{Z}_p\right)^{\vee}$. (For the precise statement, see Corollary 3.0.6) 

\smallskip

In this sense, this result \emph{distinguishes} the Mordell-Weil group and the Tate-Shafarevich group from the Selmer group.
\end{remark}

\medskip

For an estimate of the Tate-Shafarevich group at each finite layer $K_n$, we have the following theorem. This generalizes \cite[Theorem 1.10]{greenberg1999iwasawa}.

\begin{theoremintro}[Theorem 4.0.1] Suppose that $\mathrm{Coker}(S_{n}^{A})$ is finite and bounded independent of $n$, and also suppose that $\sha_{K_n}^{1}(A)_p$ is finite for all $n$. Then there exists an integer $\nu$ independent of $n$ such that $$| \sha_{K_n}^{1}(A)_p |=p^{e_n} \quad (n>>0)$$ where $$e_n=p^n \mu\left(\sha_{K_{\infty}}^{1}(A)_p^{\vee}\right) + n \lambda\left(\sha_{K_{\infty}}^{1}(A)_p^{\vee}\right) + \nu. $$
\end{theoremintro}

\medskip

\medskip

Our last main result deals with Question 1.1.5. The combination of the (cyclotomic) Iwasawa Main Conjecture with the analytic functional equation between the (conjectural) $p$-adic L-functions of $A$ and $A^t$ suggests us to expect the equality $$char_{\Lambda}\left(\mathrm{Sel}_{K_{\infty}}(A)_p^{\vee}\right)=char_{\Lambda}\left(\mathrm{Sel}_{K_{\infty}}(A^t)_p^{\vee}\right)^{\iota}$$ of characteristic ideals. Note that if the groups $\mathrm{Sel}_{K_{\infty}}(A)_p^{\vee}$ and $\mathrm{Sel}_{K_{\infty}}(A^t)_p^{\vee}$ have positive $\Lambda$-rank, the above equality of ideals is vacuous.

The below theorem refines this equality of ideals to the statement about the isomorphism classes, which is a generalization of \cite[Theorem 1.14]{greenberg1999iwasawa}.

\medskip

\begin{theoremintro}[Theorem 5.3.3] If $\mathrm{Coker}(S_{n}^{A})$ and $\mathrm{Coker}(S_{n}^{A^t})$ are finite for all $n$, then we have an isomorphism $$E\left(\mathrm{Sel}_{K_{\infty}}(A)_p^{\vee}\right)\simeq E\left(\mathrm{Sel}_{K_{\infty}}(A^t)_p^{\vee}\right)^{\iota}$$
of $\Lambda$-modules. Here $\iota$ is an involution of $\Lambda$ satisfying $\iota(T)=\frac{1}{1+T}-1$.
\end{theoremintro}

\medskip

\begin{remark}
(1) As we mentioned earlier, our result generalizes \cite[Theorem 1.14]{greenberg1999iwasawa} in two aspects. Firstly, this result is not just an equality of ideals, but rather a statement about the isomorphism classes. Secondly, we removed the $\Lambda$-cotorsion assumption of the Selmer groups $\mathrm{Sel}_{K_{\infty}}(A)_p^{\vee}$ and $\mathrm{Sel}_{K_{\infty}}(A^t)_p^{\vee}$. (See Remark 1.1.7)

\smallskip

(2) By combining Theorem C and Theorem D, we can compare the $\Lambda$-module structure of $\left(A(K_{\infty})\otimes \mathbb{Q}_p/\mathbb{Z}_p\right)^{\vee}$ with that of $\left(A^t(K_{\infty})\otimes \mathbb{Q}_p/\mathbb{Z}_p\right)^{\vee}$. The same statement holds for $\sha_{K_{\infty}}^{1}(A)_p^{\vee}$ and $\sha_{K_{\infty}}^{1}(A^t)_p^{\vee}$ also. See Proposition 5.3.4 for the precise statement. 
\end{remark}

\medskip

\subsection{Comparison with the former work}
\begin{itemize}
\item Imai \cite{imai1975remark} proved that $A\left(L(\mu_{p^{\infty}}) \right)_{\mathrm{tor}}$ is finite  where $A$ is an abelian variety over a $p$-adic local field $L$ with good reduction. Kato \cite[Page 233]{kato2004p} proved that $A\left(\mathbb{Q}_p(\mu_{p^{\infty}})\right)[p^{\infty}]$ is finite when $A$ is an elliptic curve defined over $\mathbb{Q}$ with no additional assumptions about the reduction type of $A$. Corollary 1.2.10 (Theorem 2.2.7) proves the finiteness of the $p^{\infty}$-torsion group over any ramified $\mathbb{Z}_p$-extensions (of a $p$-adic local field) when an abelian variety has good supersingular reduction. Our approach depends heavily on Theorem A and totally different with those of \cite{imai1975remark} and \cite{kato2004p}. 
\medskip
\item Theorem C is proved in \cite[Theorem 1.10]{greenberg1999iwasawa} under the additional assumption that $\mathrm{Sel}_{K_{\infty}}(A)_p^{\vee}$ is a cotorsion $\Lambda$-module. His formulation used different $\lambda$ and $\mu$, but one can show that our result implies the Greenberg's formula. See Remark 4.0.2 for this issue.
\medskip
\item In the same paper \cite{greenberg1999iwasawa}, Greenberg proved the equality of characteristic ideals between $\mathrm{Sel}_{K_{\infty}}(A)_p^{\vee}$ and $\mathrm{Sel}_{K_{\infty}}(A^t)_p^{\vee}$ under the additional assumption that $\mathrm{Sel}_{K_{\infty}}(A)_p^{\vee}$ is a cotorsion $\Lambda$-module. See \cite[Theorem 1.14]{greenberg1999iwasawa}.
\medskip
\item The technique we will use for the proof of Theorem D gives another proof of the previous algebraic functional equation results, for instance \cite[Theorem 8.2]{rubin1990main}, \cite[Proposition 1, 2, Theorem 2]{greenberg1989iwasawa}, \cite[Theorem 3.8]{jha2014algebraic}, \cite[Theorem 2.10]{jha2015functional}. Especially for \cite[Theorem 3.8]{jha2014algebraic} and \cite[Theorem 2.10]{jha2015functional}, not only we can remove the cotorsionness assumption about the Selmer groups, but also we can upgrade the statements as the comparison between elementary modules of two Selmer groups. 
\end{itemize}

\medskip

\subsection{Organization of the paper}

\medskip

\begin{itemize}
\item We record the definitions (Definition A.2.1 and the Definition A.2.8), basic properties and explicit descriptions (See Proposition A.1.6 and Proposition A.2.12) of the functors $\mathfrak{F}$ and $\mathfrak{G}$ in the appendix.
\item In section 2, we prove Theorem A (Theorem 2.1.2) and its consequence Corollary 1.2.10 (Theorem 2.2.7) by using a functor $\mathfrak{G}$. 
\item In section 3, we define the Selmer group and the Tate-Shafarevich group. We prove Theorem B (Theorem 3.0.4) and its consequence about the separating the Mordell-Weil and the Tate-Shafarevich groups (Corollary 3.0.6).
\item In section 4, we find an estimate on $|\sha_{K_n}^{1}(A)_p |$ and explain why our result is a generalization of the Greenberg's one \cite[Theorem 1.10]{greenberg1999iwasawa}.
\item In section 5, we briefly recall the Greenberg-Wiles formula \cite[Theorem 2.19]{darmon1995fermat} and the properties of the pairing of the Flach  \cite{flach1990generalisation}. Then we finally prove Theorem D (Theorem 5.3.3).
\end{itemize}

\medskip

\subsection{Notations}

We fix the notations below throughout the paper.

\smallskip

\begin{itemize}

\item We fix one rational odd prime $p$, and we let $\Lambda:=\mathbb{Z}_p[\![T]\!]$, the one variable power series ring over $\mathbb{Z}_p$. We also define $\omega_n=\omega_n(T):=(1+T)^{p^n}-1$ and $\omega_{n+1, n}:=\frac{\omega_{n+1}(T)}{\omega_n(T)}$. Note that $\omega_{n+1, n}$ is a distinguished irreducible polynomial in $\mathbb{Z}_p[\![T]\!]$.
\medskip
\item For a locally compact Hausdorff continuous $\Lambda$-module $M$, we define $M^{\vee}:=\text{Hom}_{cts}(M, \mathbb{Q}_p/\mathbb{Z}_p)$ which is also a locally compact Hausdorff. $M^{\vee}$ becomes a continuous $\Lambda$-module via the action defined by $(f\cdot \phi)(m):=\phi(\iota(f)\cdot m)$ where $f \in \Lambda$, $m \in M$, $\phi \in M^{\vee}$. We also define $M^{\iota}$ to be the same underlying set $M$ whose $\Lambda$-action is twisted by an involution $\iota
:T \rightarrow \frac{1}{1+T}-1$ of $\mathbb{Z}_p[\![T]\!]$.
\medskip
\item For a cofinitely generated $\mathbb{Z}_p$-module $X$, we define $X_{/ div}:=\frac{X}{X_{div}}$ where $X_{div}$ is the maximal $p$-divisible subgroup of $X$. Note that $\displaystyle X_{/ div} \simeq \displaystyle \lim_{\substack{\longleftarrow \\ n}}\frac{X}{p^nX}$ and $\left( X_{/div} \right)^{\vee} \simeq X^{\vee}[p^{\infty}]$. 
\medskip
\item For a finitely generated $\Lambda$-module $M$, there are prime elements $g_1, \cdots g_n$ of $\Lambda$, non-negative integer $r$, positive integers $e_1, \cdots e_n$ and a pseudo-isomorphism $\displaystyle M \rightarrow \Lambda^{r}\oplus \left(\bigoplus_{i=1}^{n}\frac{\Lambda}{g_{i}^{e_i}}\right) $. We call $\displaystyle E(M):=\Lambda^{r}\oplus \left(\bigoplus_{i=1}^{n}\frac{\Lambda}{g_{i}^{e_i}}\right)$ as \textbf{an elementary module of $M$} following \cite[Page 292]{neukirch2000cohomology}. 
\end{itemize}

\medskip

\subsection{Acknowledgements} The author would like to thank Haruzo Hida for his endless guidance and encouragement.

\medskip

\section{Structure of limit Mordell-Weil group}

\medskip

In this section, we study the structure of $\Lambda$-adic Mordell-Weil group and prove our first main theorem (Theorem 2.1.2). We first start with a lemma. Hereafter, $^{\vee}$ means the Pontryagin dual.

\medskip

\begin{lemma} Let $F$ be a finite extension of $\mathbb{Q}$ or $\mathbb{Q}_l$ for some prime $l$ and let $A$ be an abelian variety defined over $F$. Take any $\mathbb{Z}_p$-extension $F_{\infty}$ and consider the module $X:=(A(F_{\infty})[p^{\infty}])^{\vee}$.

(1) $X$ is a finitely generated torsion $\Lambda$-module with $\mu=0$, and $char_{\Lambda}X$ is coprime to $\omega_n$ for all $n$.

(2) The modules $ \frac{A(F_{\infty})[p^{\infty}]}{p^n A(F_{\infty})[p^{\infty}]} $ and $\frac{A(F_{\infty})[p^{\infty}]}{\omega_n A(F_{\infty})[p^{\infty}]} $ are finite and bounded independent of $n$.

(3) For the natural maps $$MW_{n}^{A}:A(F_n) \otimes_{\mathbb{Z}_p} \mathbb{Q}_p/\mathbb{Z}_p \rightarrow A(F_{\infty}) \otimes_{\mathbb{Z}_p} \mathbb{Q}_p/\mathbb{Z}_p[\omega_n]$$ and $$S_{n}^{A}:\mathrm{Sel}_{F_n}(A)_p \rightarrow \mathrm{Sel}_{F_{\infty}}(A)_p[\omega_n],$$ the groups $\mathrm{Ker}(MW_n^A) $ and $\mathrm{Ker}(S_n^A) $ are finite and bounded independent of $n$.
\end{lemma}

\medskip

We first remark that (2) is a direct consequence of the (1) by the structure theorem of the finitely generated $\Lambda$-modules. Hence we prove (1) and (3) only.

\medskip

\begin{proof} For (1), it suffices to show that $\frac{X}{pX}$ and $\frac{X}{\omega_nX}$ are finite. Since we have isomorphisms $$\frac{X}{pX} \simeq (A(F_{\infty})[p])^{\vee}, \quad \frac{X}{\omega_nX} \simeq (A(F_n)[p^{\infty}])^{\vee}$$ and the groups $A(F_{\infty})[p],  A(F_n)[p^{\infty}]$ are finite, we get (1).

\smallskip

For (3), by the definition of the Selmer group, we have injections $$\mathrm{Ker}(MW_n^A) \hookrightarrow \mathrm{Ker}(S_n^A) \hookrightarrow \frac{A(F_{\infty})[p^{\infty}]}{\omega_n A(F_{\infty})[p^{\infty}]}.$$ Hence the (3) follows from (2).
\end{proof}

\subsection{Proof of Theorem A}

\medskip

We first define a functor $\mathfrak{G}$. (More detailed explanation is given in the appendix.) For a finitely generated $\Lambda$-module $X$, we defined $$\displaystyle \mathfrak{G}(X):=\lim_{\substack {\longleftarrow \\ n}}\left( \frac{X}{\omega_nX}[p^{\infty}] \right).$$ We have the following description of the functor $\mathfrak{G}$. (For the proof, see the second subsection of the appendix.)

\begin{itemize}
\item $\mathfrak{G}(\Lambda)=0$. This shows that the $\mathfrak{G}(X)$ is a torsion $\Lambda$-module.
\item $\mathfrak{G}(\frac{\Lambda}{g^{e}}) \simeq \frac{\Lambda}{g^{e}}$ if $g$ is coprime to $\omega_{n}$ for all $n$.
\item \begin{equation*}
\mathfrak{G}(\frac{\Lambda}{\omega_{m+1, m}^{e}})=\left\{
\begin{array}{cc}
\frac{\Lambda}{\omega_{m+1, m}^{e-1}} & e \geq 2,\\[5pt]
0 & e=1.
\end{array}
\right.
\end{equation*}
\item $\mathfrak{G}$ is a covariant functor and preserves the pseudo-isomorphism.
\end{itemize}

\medskip

\smallskip

\begin{theorem} Let $F$ be a finite extension of $\mathbb{Q}$ or $\mathbb{Q}_p$ and let $F_{\infty}$ be any $\mathbb{Z}_p$-extension of $F$. We identify $\mathbb{Z}_p[\![\mathrm{Gal}(F_{\infty}/F)]\!]$ with $\Lambda$ by fixing a generator $\gamma$ of $\mathrm{Gal}(F_{\infty}/F)$ and identifying with $1+T.$ Then we have:

\smallskip

(1) $\mathfrak{G}\left((A(F_{\infty}) \otimes_{\mathbb{Z}_p} \mathbb{Q}_p/\mathbb{Z}_p)^{\vee}\right)=0$.

\smallskip

(2) There is a $\Lambda$-linear injection $$\displaystyle (A(F_{\infty}) \otimes_{\mathbb{Z}_p} \mathbb{Q}_p/\mathbb{Z}_p)^{\vee} \hookrightarrow \Lambda^{r} 
\oplus\left(\bigoplus_{n=1}^{t} \frac{\Lambda}{\omega_{b_{n}+1, b_{n}}} \right)$$ with finite cokernel for some integers $r, b_1, \cdots b_n$. 
\end{theorem}

\medskip

We call $A(F_{\infty}) \otimes_{\mathbb{Z}_p} \mathbb{Q}_p/\mathbb{Z}_p$ the limit Mordell-Weil group.

\medskip

\begin{remark}
(1) Hence the direct factors of $E\left((A(F_{\infty}) \otimes_{\mathbb{Z}_p} \mathbb{Q}_p/\mathbb{Z}_p)^{\vee}\right)_{\Lambda-tor}$ are only of the form $\frac{\Lambda}{\omega_{k+1, k}}$ for some $k$.

\smallskip

(2) If $F$ is a number field, then for any integer $e$ and any irreducible distinguished polynomial $h \in \Lambda$ coprime to $\omega_n$ for all $n$, the natural injection $$\sha_{F_{\infty}}^{1}(A)_p^{\vee}[h^e] \hookrightarrow \mathrm{Sel}_{F_{\infty}}(A)_p^{\vee}[h^e]$$ has a finite cokernel.

\smallskip

(3) As it will be clear in the proof, Theorem 2.1.2 holds for any $p$-adic local field $F$ also. This produces one consequence about the finiteness of the $p^{\infty}$-torsion of an abelian variety $A_{/ F}$ over $\mathbb{Z}_p$-extension. (See  Theorem 2.2.7)
\end{remark}

\medskip

Note that $(A(F_{\infty}) \otimes_{\mathbb{Z}_p} \mathbb{Q}_p/\mathbb{Z}_p)^{\vee}$ has no non-trivial finite $\Lambda$-submodule since it is $\mathbb{Z}_p$-torsion-free. Due to Proposition A.2.12, it is enough to show the assertion (1) only. As a preparation for the proof, we record the following lemma.

\medskip

\begin{lemma}\label{lemma A.1.3} (1) Let $R$ be an integral domain and $Q(R)$ be the quotient field of $R$. Consider an exact sequence of $R$-modules $0 \rightarrow A \rightarrow B \rightarrow C \rightarrow 0$ where $A$ is a $R$-torsion module. Then we have a short exact sequence $0 \rightarrow A_{R-\mathrm{tor}} \rightarrow B_{R-\mathrm{tor}} \rightarrow C_{R-\mathrm{tor}} \rightarrow 0$.

\medskip

(2) Let $0\rightarrow A \rightarrow B \rightarrow C \rightarrow 0$ be a short exact sequence of finitely generated $\mathbb{Z}_{p}$-modules. If $A$ has finite cardinality, then we have a short exact sequence $0\rightarrow A=A[p^{\infty}] \rightarrow B[p^{\infty}] \rightarrow C[p^{\infty}] \rightarrow 0$.

\medskip

(3) If $0 \rightarrow X \rightarrow Y \rightarrow  Z \rightarrow W \rightarrow 0$ is an exact sequence of cofinitely generated $\mathbb{Z}_p$-modules with finite $W$, then the sequence $  X_{ / div} \rightarrow  Y_{/ div} \rightarrow  Z_{/ div} \rightarrow W \rightarrow 0$ is exact. 
\end{lemma}

\begin{proof} We prove (1) and (3) only, since (2) is a direct consequence of (1). Note that for an $R$-module $M$, we have $\text{Tor}_{1}^{R}(M, Q(R)/R) \simeq M_{R-\mathrm{tor}}$. Since $A$ is a $R$-torsion module, we have $A \otimes_{R}Q(R)=0$. Now (1) follows from the long exact sequence associated with functor $\text{Tor}_{1}^{R}(-,Q(R)/R)$.

\smallskip

For (3), break up the sequence to $0 \rightarrow X \rightarrow Y \rightarrow T \rightarrow 0 $ and $0 \rightarrow T \rightarrow  Z \rightarrow W \rightarrow 0 $. 
Applying $p^{\infty}$-torsion functor to the Pontryagin dual of the first short exact sequence gives $X_{/ div} \rightarrow Y_{/ div} \rightarrow T_{/ div} \rightarrow 0  $. By (2), we get $0 \rightarrow T_{/ div} \rightarrow  Z_{/ div} \rightarrow W_{/ div} \rightarrow 0$. Combining these two sequences proves the assertion.
\end{proof}

\medskip

\begin{proof}[Proof of $\mathrm{Theorem}$ $2.1.2$] Let $C_n$ be the cokernel of the natural map $$MW_{n}^{A}:A(F_n) \otimes_{\mathbb{Z}_p} \mathbb{Q}_p/\mathbb{Z}_p \rightarrow A(F_{\infty}) \otimes_{\mathbb{Z}_p} \mathbb{Q}_p/\mathbb{Z}_p[\omega_n].$$ Since $\displaystyle \lim_{\substack {\longrightarrow \\ n}}A(F_n) \otimes_{\mathbb{Z}_p} \mathbb{Q}_p/\mathbb{Z}_p=A(F_{\infty}) \otimes_{\mathbb{Z}_p} \mathbb{Q}_p/\mathbb{Z}_p$ by definition, we get $\displaystyle \lim_{\substack {\longrightarrow \\ n}}C_n=0$ and $\displaystyle \lim_{\substack {\longrightarrow \\ n}}(C_n)_{/div}=0$. \\

From an exact sequence $$ A(F_n) \otimes_{\mathbb{Z}_p} \mathbb{Q}_p/\mathbb{Z}_p \rightarrow A(F_{\infty}) \otimes_{\mathbb{Z}_p} \mathbb{Q}_p/\mathbb{Z}_p[\omega_n] \rightarrow C_n \rightarrow 0,$$ we have $$ 0=(A(F_n) \otimes_{\mathbb{Z}_p} \mathbb{Q}_p/\mathbb{Z}_p)_{/div} \rightarrow (A(F_{\infty}) \otimes_{\mathbb{Z}_p} \mathbb{Q}_p/\mathbb{Z}_p[\omega_n])_{/div} \rightarrow (C_n)_{/div} \rightarrow 0$$ by Lemma 2.1.4-(3). Taking direct limit to this sequence shows $$\displaystyle \lim_{\substack {\longrightarrow \\ n}} (A(F_{\infty}) \otimes_{\mathbb{Z}_p} \mathbb{Q}_p/\mathbb{Z}_p[\omega_n])_{/div}=0$$ since $\displaystyle \lim_{\substack {\longrightarrow \\ n}}(C_n)_{/div}=0$. If we take the Pontryagin dual, we get

$$\displaystyle \lim_{\substack {\longleftarrow \\ n}} \frac{(A(F_{\infty}) \otimes_{\mathbb{Z}_p} \mathbb{Q}_p/\mathbb{Z}_p)^{\vee}}{\omega_n  (A(F_{\infty}) \otimes_{\mathbb{Z}_p} \mathbb{Q}_p/\mathbb{Z}_p)^{\vee}}[p^{\infty}]=0.$$

Hence we get
$$\displaystyle \mathfrak{G}\left((A(F_{\infty}) \otimes_{\mathbb{Z}_p} \mathbb{Q}_p/\mathbb{Z}_p)^{\vee}\right):= \lim_{\substack {\longleftarrow \\ n}} \frac{(A(F_{\infty}) \otimes_{\mathbb{Z}_p} \mathbb{Q}_p/\mathbb{Z}_p)^{\vee}}{\omega_n  (A(F_{\infty}) \otimes_{\mathbb{Z}_p} \mathbb{Q}_p/\mathbb{Z}_p)^{\vee}}[p^{\infty}]=0.$$
\end{proof}

\smallskip

Next we state the control result of the limit Mordell-Weil group under the $\Lambda$-cotorsion assumption. For the second statement about the $\lambda$-invariant, we need to use Theorem 2.1.2.\\

\begin{theorem} Let $F$ be a number field, $F_{\infty}$ be a $\mathbb{Z}_p$-extension of $F$ and $F_n$ be the $n$-th layer. The following two assertions are equivalent:

\begin{itemize}
\item The sequence $\lbrace\mathrm{rank}_{\mathbb{Z}}A(F_n)\rbrace_{n \geq 0}$ is bounded.
\item $A(F_{\infty}) \otimes_{\mathbb{Z}_p} \mathbb{Q}_p/\mathbb{Z}_p$ is a cotorsion $\Lambda$-module.
\end{itemize}

\smallskip

If this equivalent condition holds, the natural map $$A(F_n) \otimes_{\mathbb{Z}_p} \mathbb{Q}_p/\mathbb{Z}_p \rightarrow A(F_{\infty}) \otimes_{\mathbb{Z}_p} \mathbb{Q}_p/\mathbb{Z}_p[\omega_n]$$ is surjective for almost all $n$ and $\mathrm{rank}_{\mathbb{Z}} A(F_{n})$ stabilizes to the $\lambda \left( (A(F_{\infty}) \otimes_{\mathbb{Z}_p} \mathbb{Q}_p/\mathbb{Z}_p)^{\vee}\right)$.
\end{theorem}

\medskip

\begin{proof} By Lemma 2.0.1, the natural map $A(F_n) \otimes_{\mathbb{Z}_p} \mathbb{Q}_p/\mathbb{Z}_p \rightarrow A(F_{\infty}) \otimes_{\mathbb{Z}_p} \mathbb{Q}_p/\mathbb{Z}_p[\omega_n]$ has finite kernel for all $n$. If $A(F_{\infty}) \otimes_{\mathbb{Z}_p} \mathbb{Q}_p/\mathbb{Z}_p$ is a cotorsion $\Lambda$-module, then by the $\Lambda$-module theory, $\mathbb{Z}_p$-corank of modules $A(F_{\infty}) \otimes_{\mathbb{Z}_p} \mathbb{Q}_p/\mathbb{Z}_p[\omega_n]$ are bounded, and hence $\mathrm{rank}_{\mathbb{Z}}A(F_n)$ is bounded.

\smallskip

Conversely, assume that $\mathrm{rank}_{\mathbb{Z}}A(F_n)$ is bounded and take $n$ so that $$\mathrm{rank}_{\mathbb{Z}}A(F_{n+k})=\mathrm{rank}_{\mathbb{Z}}A(F_n)$$ for all $k \geq 0$. Now consider the following diagram:

\begin{center}
\begin{tikzcd}
A(F_n) \otimes_{\mathbb{Z}_p} \mathbb{Q}_p/\mathbb{Z}_p \arrow[r, "r"] \arrow[d, "s"] & A(F_{n+1}) \otimes_{\mathbb{Z}_p} \mathbb{Q}_p/\mathbb{Z}_p \arrow[d, "t"] \\
A(F_{\infty}) \otimes_{\mathbb{Z}_p} \mathbb{Q}_p/\mathbb{Z}_p[\omega_n] \arrow[r, hook] & A(F_{\infty}) \otimes_{\mathbb{Z}_p} \mathbb{Q}_p/\mathbb{Z}_p[\omega_{n+1}]
\end{tikzcd}
\end{center}

Since the map $s$ has finite kernel, $\mathrm{Ker}(r)$ is also finite. Considering the $\mathbb{Z}_p$-corank shows that $r$ is surjective. By taking direct limit, we get a surjection $$\displaystyle  A(F_n) \otimes_{\mathbb{Z}_p} \mathbb{Q}_p/\mathbb{Z}_p \twoheadrightarrow A(F_{\infty}) \otimes_{\mathbb{Z}_p} \mathbb{Q}_p/\mathbb{Z}_p.$$

Since this map factors through the natural inclusion $$A(F_{\infty}) \otimes_{\mathbb{Z}_p} \mathbb{Q}_p/\mathbb{Z}_p[\omega_n] \hookrightarrow A(F_{\infty}) \otimes_{\mathbb{Z}_p} \mathbb{Q}_p/\mathbb{Z}_p,$$ the above inclusion is an isomorphism indeed, and hence $A(F_{\infty}) \otimes_{\mathbb{Z}_p} \mathbb{Q}_p/\mathbb{Z}_p$ is a cotorsion $\Lambda$-module. This also proves the assertion about the stabilized value of the sequence $\lbrace\mathrm{rank}_{\mathbb{Z}}A(F_n) \rbrace_{n \geq 0}$.\end{proof}

\begin{remark} (1) In above theorem, the condition that $A(F_{\infty}) \otimes_{\mathbb{Z}_p} \mathbb{Q}_p/\mathbb{Z}_p$ is a cotorsion $\Lambda$-module is \emph{not} enough to guarantee that $$\mathrm{rank}_{\mathbb{Z}} A(F_{n})=\lambda \left( (A(F_{\infty}) \otimes_{\mathbb{Z}_p} \mathbb{Q}_p/\mathbb{Z}_p)^{\vee}\right) \quad (n>>0).$$ Indeed, we need to use Theorem 2.1.2 that the direct factors of $E\left((F(K_{\infty}) \otimes_{\mathbb{Z}_p} \mathbb{Q}_p/\mathbb{Z}_p)^{\vee}\right)_{\Lambda-tor}$ are only of the form $\frac{\Lambda}{\omega_{k+1, k}}$ for some $k$.

\smallskip

(2) If $F$ is a number field, consider a $\Lambda$-linear injection $$\displaystyle (A(F_{\infty}) \otimes_{\mathbb{Z}_p} \mathbb{Q}_p/\mathbb{Z}_p)^{\vee} \hookrightarrow \Lambda^{r} 
\oplus \left(\bigoplus_{n=1}^{t} \frac{\Lambda}{\omega_{b_{n}+1, b_{n}}} \right) $$ in Theorem 2.1.2. Assume that $\mathrm{Coker}(S_{n}^{A})$ and $\sha_{F_n}^{1}(A)_p$ are finite for all $n$. Then by the snake lemma, the natural map $$ A(F_n) \otimes_{\mathbb{Z}_p} \mathbb{Q}_p/\mathbb{Z}_p \rightarrow A(F_{\infty}) \otimes_{\mathbb{Z}_p} \mathbb{Q}_p/\mathbb{Z}_p[\omega_n] $$ has finite kernel and cokernel for all $n$. By comparing the $\mathbb{Z}_p$-coranks of the two modules, we get $$\displaystyle r=\lim_{\substack{n \longrightarrow \infty}}\frac{\mathrm{rank}_{\mathbb{Z}}A(F_n)}{p^n}, \quad a_{n}=\frac{\mathrm{rank}_{\mathbb{Z}}A(F_n)-\mathrm{rank}_{\mathbb{Z}}A(F_{n-1})}{p^{n-1}(p-1)}-r \quad (n \geq 1)$$ where $a_n$ is defined as the number of $1 \leq i \leq t$ satisfying $b_i=n-1$. Hence for this case, we can describe $E\left((A(F_{\infty}) \otimes_{\mathbb{Z}_p} \mathbb{Q}_p/\mathbb{Z}_p)^{\vee}\right)$ by using the sequence $\lbrace\mathrm{rank}_{\mathbb{Z}}A(F_n) \rbrace_{n \geq 0}$.
\end{remark}

\medskip

\subsection{Consequence of Theorem A}

We prove one interesting consequence of Theorem 2.1.2 about the finiteness of the $p^{\infty}$-torsion group of an abelian variety over a $p$-adic local field. Let $L$ be a finite extension of $\mathbb{Q}_p$, $L_{\infty}$ be a $\mathbb{Z}_p$-extension of $L$, $L_n$ be a $n$-th layer, and $A$ be an abelian variety defined over $L$.

We consider $A(L_{\infty})[p^{\infty}]$, which is a cofinitely generated cotorsion $\Lambda$-module by Lemma 2.0.1. If $A$ has (potentially) good reduction over $L$, Imai \cite{imai1975remark} proved that the torsion subgroup of the $A\left( L(\mu_{p^{\infty}}) \right)$ is finite. By using Theorem 2.1.2, we prove that if $A$ has potentially supersingular reduction, $A(L_{\infty})[p^{\infty}]$ is finite for the general ramified $\mathbb{Z}_p$-extension $L_{\infty}/L$.

\smallskip

\begin{theorem} If $A_{/ L}$ has potentially supersingular reduction and $L_{\infty}/L$ is a ramified $\mathbb{Z}_p$-extension, then $A(L_{\infty})[p^{\infty}]$ is finite.
\end{theorem}

\begin{proof} We may assume that $A$ has supersingular reduction over $L$. Let $\mathcal{F}$ be a formal group of $A$. Since $A_{/ L}$ has supersingular reduction, we have an isomorphism $$H^{1}(L_{\infty}, \mathcal{F}) \simeq H^{1}(L_{\infty}, A)[p^{\infty}].$$ Since $L_{\infty}/L$ is a ramified $\mathbb{Z}_p$-extension, by \cite[Proposition 2.10, Theorem 2.13]{coates1996kummer} we have $$H^{1}(L_{\infty}, \mathcal{F})=0.$$ Hence from the Kummer sequence, we have an isomorphism $$H^{1}(L_{\infty}, A[p^{\infty}])^{\vee} \simeq (A(L_{\infty}) \otimes_{\mathbb{Z}_p} \mathbb{Q}_p/\mathbb{Z}_p)^{\vee}.$$ Since $\Lambda$-torsion part of $H^{1}(L_{\infty}, A[p^{\infty}])^{\vee}$ is pseudo-isomorphic to $(A^t(L_{\infty})[p^{\infty}])^{\vee}$ by \cite[Proposition 3.1]{greenberg1989iwasawa} (up to twisting by $\iota$), combining Theorem 2.1.2 and Lemma 2.0.1 shows the desired assertion.
\end{proof}

\smallskip

\begin{remark} (1) In the proof of the above Theorem, we have also proved that $ (A(L_{\infty}) \otimes_{\mathbb{Z}_p} \mathbb{Q}_p/\mathbb{Z}_p)^{\vee} $ is a torsion-free $\Lambda$-module.

\smallskip

(2) Kato \cite[Page 233]{kato2004p} proved that $A\left(\mathbb{Q}_p(\mu_{p^{\infty}})\right)[p^{\infty}]$ is finite when $A$ is an elliptic curve defined over $\mathbb{Q}$. By using Theorem 2.2.7, we can generalize this result to the finiteness of $A(L_{\infty})[p^{\infty}]$ where $A$ is an elliptic curve with potentially good reduction and $L_{\infty}/L$ is a any ramified $\mathbb{Z}_p$-extension.

The proof of the potentially supersingular case follows from Theorem 2.2.7. For the potentially ordinary case, one can use the filtration on the Tate module $T_pA$. 
\end{remark}

\medskip

\section{Selmer and $\sha^{1}$}

\medskip

We study the Selmer group and the Tate-Shafarevich group in this section. We will describe the $\Lambda$-adic Tate-Shafarevich group by the $\Lambda$-adic Selmer group under mild assumptions. (See Theorem 3.0.4)

\medskip

\begin{definition}\label{definition 4.1.3} Let $F$ be a number field and $A$ be an abelian variety over $F$. Let $S$ be a finite set of places of $F$ containing the places over $p$, infinite places and places of bad reductions of $A$. We define the Selmer group and the Tate-Shafarevich group as follows:

(1) $\displaystyle \mathrm{Sel}_{F}(A)_p:=\mathrm{Ker}\left(H^{1}(F^S/F, A[p^{\infty}]) \rightarrow \prod_{\substack{v \in S}}H^{1}(F_v, A)\right)
$.

(2) $\displaystyle \sha_{F}^{1}(A)_p:=\mathrm{Ker}\left(H^{1}(F^S/F, A)[p^{\infty}] \rightarrow
\prod_{\substack{v \in S}}H^{1}(F_v, A)[p^{\infty}]\right)$.
\end{definition}

\begin{remark}\label{remark 4.1.4} By \cite[Corollary $\mathrm{I}.6.6$]{milne2006arithmetic}, this definition is independent of the choice of $S$ as long as $S$ contains infinite places, primes over $p$ and primes of bad reduction of $A$. Moreover, all modules in the above definition are cofinitely generated $\mathbb{Z}_p$-modules.
\end{remark}

\medskip

\begin{notation} Hereafter, 

(1) we let $A$ be an abelian variety over a number field $K$ and let $S$ be a finite set of places of $K$ containing the places over $p$, infinite places and places of bad reductions of $A$. We fix a $\mathbb{Z}_p$-extension $K_{\infty}$ of $K$.

\smallskip

(2) we identify $\mathbb{Z}_p[\![\mathrm{Gal}(K_{\infty}/K)]\!]$ with $\Lambda$ by fixing a generator $\gamma$ of $\mathrm{Gal}(K_{\infty}/K)$ and identifying with $1+T.$

\smallskip

(3) we let $$S_{n}^{A}:\mathrm{Sel}_{K_n}(A)_p \rightarrow \mathrm{Sel}_{K_{\infty}}(A)_p[\omega_n]$$ be the natural restriction map.
\end{notation} 

\medskip

Next we state the result that under the control of Selmer groups and Tate-Shafarevich conjecture, $\sha_{K_{\infty}}^{1}(A)_p$ is a cotorsion $\Lambda$-module. This can be regarded as a $\Lambda$-adic analogue of Tate-Shafarevich conjecture. (Recall that $\displaystyle \mathfrak{G}(X)=\lim_{\substack{\longleftarrow \\ n}}\left(\frac{X}{\omega_{n}X}[p^{\infty}]\right)$.) \\

\begin{theorem} If $\mathrm{Coker}(S_{n}^{A})$ and $\sha_{K_n}^{1}(A)_p$ are finite for all $n$, then we have an isomorphism $$\sha_{K_{\infty}}^{1}(A)_p^{\vee} \simeq \mathfrak{G}\left(\mathrm{Sel}_{K_{\infty}}(A)_p^{\vee}\right)$$ of $\Lambda$-modules. In particular, $\sha_{K_{\infty}}^{1}(A)_p$ is a cotorsion $\Lambda$-module. 
\end{theorem}

\medskip

Hence under the Tate-Shafarevich conjecture, if the control theorem for the Selmer groups holds, then we can describe the $\Lambda$-module structure of $\sha_{K_{\infty}}^{1}(A)_p$ by using $\mathrm{Sel}_{K_{\infty}}(A)_p$.

\medskip

\begin{remark} (1) Since we have an explicit description of the functor $\mathfrak{G}$ in the appendix (Proposition A.2.12), if we have data of $E\left(\mathrm{Sel}_{K_{\infty}}(A)_p^{\vee}\right)$, we can disassemble the factors of $(A(K_{\infty}) \otimes_{\mathbb{Z}_p} \mathbb{Q}_p/\mathbb{Z}_p)^{\vee}$ and the factors of $\sha_{K_{\infty}}^{1}(A)_p^{\vee}$ from $E\left(\mathrm{Sel}_{K_{\infty}}(A)_p^{\vee}\right)$ completely. See Corollary 3.0.6 below.

\smallskip

(2) If we make the assumption only about the $\mathrm{Coker}(S_{n}^{A})$, then we have the similar statement for the $\Lambda$-adic Bloch-Kato's Tate-Shafarevich group (which is defined to be finite at finite layers $K_n$). This remark can be applied to Corollary 3.0.6, Theorem 4.0.1, Proposition 5.3.4, and Theorem 5.3.5.
\end{remark}

\medskip

\begin{proof}[Proof of $\mathrm{Theorem}$ $3.0.4$] We start from the natural map $S_{n}^{A}:\mathrm{Sel}_{K_n}(A)_p \rightarrow \mathrm{Sel}_{K_{\infty}}(A)_p[\omega_n]$. Note that $\displaystyle \lim_{\substack{\longrightarrow \\ n}}\mathrm{Ker}(S_n^{A})=\lim_{\substack{\longrightarrow \\ n}}\mathrm{Coker}(S_n^{A})=0$ by definition. Hence we get $$\displaystyle \lim_{\substack{\longleftarrow \\ n}}\mathrm{Ker}(S_n^{A})^{\vee}[p^{\infty}]=\lim_{\substack{\longleftarrow \\ n}}\mathrm{Coker}(S_n^{A})^{\vee}[p^{\infty}]=0.$$

\smallskip

On the other hand, since $\mathrm{Coker}(S_n^{A})$ is finite, by Lemma 2.1.4-(3) we get an exact sequence $$0 \rightarrow \mathrm{Coker}(S_n^{A})^{\vee}[p^{\infty}] \rightarrow \frac{\mathrm{Sel}_{K_{\infty}}(A)_p^{\vee}}{\omega_n \mathrm{Sel}_{K_{\infty}}(A)_p^{\vee}}[p^{\infty}] \rightarrow \mathrm{Sel}_{K_n}(A)_p^{\vee}[p^{\infty}] \rightarrow \mathrm{Ker}(S_n^{A})^{\vee}[p^{\infty}]$$  where $\mathrm{Sel}_{K_n}(A)_p^{\vee}[p^{\infty}]$ is isomorphic to $\sha_{K_n}^{1}(A)_p^{\vee}$ since $\sha_{K_n}^{1}(A)_p$ is finite. Now taking projective limit to the above sequence gives

$$\displaystyle \mathfrak{G}\left(\mathrm{Sel}_{K_{\infty}}(A)_p^{\vee}\right) := \lim_{\substack {\longleftarrow \\ n}}\frac{\mathrm{Sel}_{K_{\infty}}(A)_p^{\vee}}{\omega_n \mathrm{Sel}_{K_{\infty}}(A)_p^{\vee}}[p^{\infty}] \simeq \sha_{K_{\infty}}^{1}(A)_p^{\vee}. $$
\end{proof}

\smallskip

Under the same conditions with the above theorem, we can \emph{separate} the Mordell-Weil group and the Tate-Shafarevich group from the Selmer group. More precisely, we describe the elementary modules of the Mordell-Weil group and the Tate-Shafarevich group by using the elementary module of the Selmer group. 

\medskip

\begin{cor} Suppose that $\mathrm{Coker}(S_{n}^{A})$ and $\sha_{K_n}^{1}(A)_p$ are finite for all $n$ and let $$\displaystyle E\left(\mathrm{Sel}_{K_{\infty}}(A)_p^{\vee}\right) \simeq \Lambda^{r} 
\oplus \left(\bigoplus_{i=1}^{d} \frac{\Lambda}{g_{i}^{l_{i}}}\right) 
\oplus \left(\bigoplus_{\substack{m=1 \\ e_{1} \cdots e_{f} \geq 2}}^{f} \frac{\Lambda}{\omega_{a_{m}+1, a_{m}}^{e_{m}}} \right)
\oplus \left(\bigoplus_{n=1}^{t} \frac{\Lambda}{\omega_{b_{n}+1, b_{n}}}\right) $$
where $r \geq 0$, $g_{1}, \cdots g_{d}$ are prime elements of $\Lambda$ which are coprime to $\omega_{n}$ for all $n$, $d \geq 0$, $l_{1}, \cdots, l_{d} \geq 1$, $f \geq 0$, $e_{1}, \cdots, e_{f} \geq 2$ and $t \geq 0$. Then we have isomorphisms 

$$ \displaystyle E\left(\sha_{K_{\infty}}^{1}(A)_p^{\vee}\right) \simeq  \left(\bigoplus_{i=1}^{d} \frac{\Lambda}{g_{i}^{l_{i}}}\right) 
\oplus \left(\bigoplus_{\substack{m=1 \\ e_{1} \cdots e_{f} \geq 2}}^{f} \frac{\Lambda}{\omega_{a_{m}+1, a_{m}}^{e_{m}-1}} \right)$$ and $$ \displaystyle E\left((A(K_{\infty}) \otimes_{\mathbb{Z}_p} \mathbb{Q}_p/\mathbb{Z}_p)^{\vee}\right) \simeq \Lambda^{r} 
\oplus \left(\bigoplus_{\substack{m=1 \\ e_{1} \cdots e_{f} \geq 2}}^{f} \frac{\Lambda}{\omega_{a_{m}+1, a_{m}}} \right)
\oplus \left(\bigoplus_{n=1}^{t} \frac{\Lambda}{\omega_{b_{n}+1, b_{n}}}\right).$$
\end{cor}

\begin{proof} The isomorphism for $\sha_{K_{\infty}}^{1}(A)_p^{\vee}$ is a direct consequence of Theorem 3.0.4. For the second isomorphism, since $\sha_{K_{\infty}}^{1}(A)_p^{\vee}$ is a torsion $\Lambda$-module by Theorem 3.0.4, the exact sequence $$0 \rightarrow \sha_{K_{\infty}}^{1}(A)_p^{\vee} \rightarrow \mathrm{Sel}_{K_{\infty}}(A)_p^{\vee} \rightarrow (A(K_{\infty}) \otimes_{\mathbb{Z}_p} \mathbb{Q}_p/\mathbb{Z}_p)^{\vee} \rightarrow 0$$ induces another exact sequence $$ 0 \rightarrow \left(\sha_{K_{\infty}}^{1}(A)_p^{\vee} \right) _{\Lambda-\mathrm{tor}} \rightarrow \left(\mathrm{Sel}_{K_{\infty}}(A)_p^{\vee}\right)_{\Lambda-\mathrm{tor}} \rightarrow \left(A(K_{\infty}) \otimes_{\mathbb{Z}_p} \mathbb{Q}_p/\mathbb{Z}_p\right)^{\vee}_{\Lambda-\mathrm{tor}} \rightarrow 0 $$ by Lemma 2.1.4-(1). Now Theorem 3.0.4 shows the desired isomorphism.
\end{proof}

\medskip

\section{Estimates on the size of $\sha^{1}$}

Now we find an estimate of $|\sha_{K_n}^{1}(A)_p|$, which is an analogue of the Iwasawa's class number formula \cite[Theorem 1.1]{greenberg1999iwasawa}.

\medskip

\begin{theorem} Suppose that $\mathrm{Coker}(S_{n}^{A})$ is finite and bounded independent of $n$ and also suppose that $\sha_{K_n}^{1}(A)_p$ is finite for all $n$. Then there exists an integer $\nu$ independent of $n$ such that $$| \sha_{K_n}^{1}(A)_p |=p^{e_n} \quad (n>>0)$$ where $$e_n=p^n \mu\left(\sha_{K_{\infty}}^{1}(A)_p^{\vee}\right) + n \lambda\left(\sha_{K_{\infty}}^{1}(A)_p^{\vee}\right) + \nu. $$
\end{theorem}

\medskip

\begin{remark}
(1) This theorem can be an evidence for the control of the Tate-Shafarevich group over the tower of fields $\lbrace K_n \rbrace_{n \geq 0}$. More precisely, if the characteristic ideal of $ \sha_{K_{\infty}}^{1}(A)_p^{\vee} $ is coprime to $\omega_n$ for all $n$, (Note that by Theorem 3.0.4, $ \sha_{K_{\infty}}^{1}(A)_p^{\vee}$ is a $\Lambda$-torsion under the assumption of Theorem 4.0.1)  then by the $\Lambda$-module theory, we get $$| \sha_{K_{\infty}}^{1}(A)_p[\omega_n] |=p^{e_n}$$ for the same $e_n$ appearing in the above theorem. Hence we could expect the natural map $$\sha_{K_n}^{1}(A)_p \rightarrow \sha_{K_{\infty}}^{1}(A)_p[\omega_n]$$ to have bounded kernel and cokernel.

\smallskip

(2) This theorem generalizes \cite[Theorem 1.10]{greenberg1999iwasawa} by removing the cotorsionness assumption of $\mathrm{Sel}_{K_{\infty}}(A)_p$. If Selmer group is a $\Lambda$-cotorsion, then our formula recovers that of \cite[Theorem 1.10]{greenberg1999iwasawa}. We give a brief explanation here. 

\smallskip

In \cite[Theorem 1.10]{greenberg1999iwasawa}, the Greenberg's formula was $$| \sha_{K_n}^{1}(A)_p |=p^{f_n} \quad (n>>0)$$ for $f_n=p^n \mu_{A} + n \cdot (\lambda_{A} - \lambda_{A}^{MW}) + \nu$ where 

\begin{itemize}
\item $\mu_{A}$ is a $\mu$-invariant of the Selmer group $\mathrm{Sel}_{K_{\infty}}(A)_p^{\vee}$.
\item $\lambda_{A}$ is a $\lambda$-invariant of the Selmer group $\mathrm{Sel}_{K_{\infty}}(A)_p^{\vee}$.
\item $\lambda_{A}^{MW}$ is the stabilized value of $\lbrace\mathrm{rank}_{\mathbb{Z}}A(K_n)\rbrace_{n \geq 0}$.
\end{itemize}

\smallskip

Since $\left(A(K_{\infty}) \otimes \mathbb{Q}_p/\mathbb{Z}_p \right)^{\vee}$ is a $\Lambda$-torsion for this case, $\lambda_{A}^{MW}$ is same as the $\lambda$-invariant of $\left(A(K_{\infty}) \otimes \mathbb{Q}_p/\mathbb{Z}_p \right)^{\vee}$ by Theorem 2.1.5. Hence by the multiplicative property of characteristic ideals, we get $$\lambda_{A} - \lambda_{A}^{MW}=\lambda\left(\sha_{K_{\infty}}^{1}(A)_p^{\vee}\right).$$

For the $\mu$-invariant, by Theorem 2.1.2, $\left(A(K_{\infty}) \otimes \mathbb{Q}_p/\mathbb{Z}_p \right)^{\vee}$ has $\mu$-invariant zero. Hence we get $$\mu_{A}=\mu\left(\sha_{K_{\infty}}^{1}(A)_p^{\vee}\right)$$ which justifies the claim.\end{remark}

\medskip

\begin{proof}[Proof of $\mathrm{Theorem}$ 4.0.1] By the straight forward calculation, we can check that for a finitely generated $\Lambda$-module $Y$, we have $$\mathrm{log}_{p} |\frac{Y}{\omega_nY}[p^{\infty}]|=p^n \mu \left(\mathfrak{G}(Y) \right) + n \lambda \left(\mathfrak{G}(Y) \right) + \nu$$ for all $n>>0$.

\medskip

If we let $Y=\mathrm{Sel}_{K_{\infty}}(A)_p^{\vee}$, this gives an estimate for the group $\frac{\mathrm{Sel}_{K_{\infty}}(A)_p^{\vee}}{\omega_n \mathrm{Sel}_{K_{\infty}}(A)_p^{\vee}}[p^{\infty}] $. Since $\mathrm{Ker}(S^A_n)$ and $\mathrm{Coker}(S^A_n)$ are bounded independent of $n$, we get the desired assertion since we have isomorphisms $$\mathrm{Sel}_{K_n}(A)_p^{\vee}[p^{\infty}] \simeq \sha_{K_n}^{1}(A)_p^{\vee} \quad \mathrm{and} \quad \sha_{K_{\infty}}^{1}(A)_p^{\vee} \simeq \mathfrak{G}\left(\mathrm{Sel}_{K_{\infty}}(A)_p^{\vee}\right).$$ The second isomorphism holds due to Theorem 3.0.4.
\end{proof}

\medskip

\section{Algebraic functional equation}

\medskip

In this section, we want to compare two modules $E\left(\mathrm{Sel}_{K_{\infty}}(A)_p^{\vee}\right)$ and  $E\left(\mathrm{Sel}_{K_{\infty}}(A^t)_p^{\vee}\right)^{\iota}$ under the control of the Selmer groups of $A$ and $A^t$. The strategy of the proof is the following:

\begin{itemize}
\item By using the Greenberg-Wiles formula \cite[Theorem 2.19]{darmon1995fermat}, we compare the $\mathbb{Z}_p$-corank of Selmer groups of $A$ and $A^t$ at each finite layer $K_n$. (See section 5.1 below)
\smallskip
\item Using two functors $\mathfrak{F}$ and $\mathfrak{G}$, we can lift the duality between Selmer groups of $A$ and $A^t$ (induced by the Flach's pairing: See section 5.2) to the $\Lambda$-adic setting. 
\end{itemize}

\medskip

\subsection{The Greenberg-Wiles formula} We recall the Greenberg-Wiles formula in \cite[Theorem 2.19]{darmon1995fermat}, which compares the cardinalities of two finite Selmer groups. We define $\mathrm{Sel}_{K_n, p^m}(A)$ as the kernel of the natural restriction map $$ \displaystyle  H^1(K^S/K_n, A[p^m]) \rightarrow \prod_{\substack{v }}H^{1}(K_n, A)$$ where $v$ runs through the primes of $K_n$ over the primes in $S$.

\medskip

\begin{cor}\label{cor 8.1.3} (1) (Greenberg-Wiles formula) For a fixed $n$, $\frac{| \mathrm{Sel}_{K_n, p^m}(A) |}{| \mathrm{Sel}_{K_n, p^m}(A^t) |}$ becomes stationary as $m \rightarrow \infty$.

\medskip

(2) We have $\mathrm{corank}_{\mathbb{Z}_p} \mathrm{Sel}_{K_n}(A)_p=\mathrm{corank}_{\mathbb{Z}_p} \mathrm{Sel}_{K_n}(A^t)_p$ for all $n$.
\end{cor}

\medskip

\begin{proof} For the proof of (1), see \cite[Theorem 2.19]{darmon1995fermat}. For (2), consider the following two natural exact sequences :
\begin{align*}
0 \rightarrow \frac{A(K_n)[p^{\infty}]}{p^mA(K_n)[p^{\infty}]} \rightarrow \mathrm{Sel}_{K_n, p^m}(A) \rightarrow \mathrm{Sel}_{K_n}(A)_p[p^m] \rightarrow 0 \\
0 \rightarrow \frac{A^t(K_n)[p^{\infty}]}{p^mA^t(K_n)[p^{\infty}]} \rightarrow \mathrm{Sel}_{K_n, p^m}(A^t) \rightarrow \mathrm{Sel}_{K_n}(A^t)_p[p^m] \rightarrow 0
\end{align*}

Since $A(K_n)[p^{\infty}]$ and $A^t(K_n)[p^{\infty}]$ are finite groups, the size of the groups $\frac{A(K_n)[p^{\infty}]}{p^mA(K_n)[p^{\infty}]}, \frac{A^t(K_n)[p^{\infty}]}{p^mA^t(K_n)[p^{\infty}]}$ are of bounded order as $m$ varies. If we consider the ratio between two groups $ \mathrm{Sel}_{K_n, p^m}(A) $ and $ \mathrm{Sel}_{K_n, p^m}(A^t) $, we get the (2) from (1).
\end{proof}

\medskip

\subsection{Flach's pairing on Selmer groups}

We briefly recall properties of the pairing of Flach.

\medskip

For any finite extension $\mathfrak{k}$ of $K$ contained in $K^{S}$, Flach (\cite{flach1990generalisation}) constructed a $\mathrm{Gal}(\mathfrak{k}/K)$-equivariant bilinear pairing $$F_{\mathfrak{k}}:\mathrm{Sel}_{\mathfrak{k}}(A)_p \times \mathrm{Sel}_{\mathfrak{k}}(A^t)_p \rightarrow \mathbb{Q}_p/\mathbb{Z}_p$$ whose left kernel (resp. right kernel) is the maximal $p$-divisible subgroup of $\mathrm{Sel}_{\mathfrak{k}}(A)_p$ (resp. $\mathrm{Sel}_{\mathfrak{k}}(A^t)_p$). Here $\mathrm{Gal}(\mathfrak{k}/K)$-equivariance means the property $$ F_{\mathfrak{k}}(g \cdot x, g \cdot y)=F_{\mathfrak{k}}(x, y)  $$ for all $g \in \mathrm{Gal}(\mathfrak{k}/K)$ and $x \in \mathrm{Sel}_{\mathfrak{k}}(A)_p, y \in \mathrm{Sel}_{\mathfrak{k}}(A^t)_p$.

\medskip

If we have two finite extensions $\mathfrak{k}_1 \geq \mathfrak{k}_2$ of $K$ in $K^S$, we have the following functorial diagram:

\smallskip

\begin{center}
\begin{tikzcd}
\mathrm{Sel}_{\mathfrak{k}_2}(A)_p   & \times & \mathrm{Sel}_{\mathfrak{k}_2}(A^t)_p \arrow[dd, "\mathrm{Res}"] \arrow[rrd, "F_{\mathfrak{k}_2}"] &  &  \\
 &  &  &  & \mathbb{Q}_p/\mathbb{Z}_p \\
\mathrm{Sel}_{\mathfrak{k}_1}(A)_p \arrow[uu, "\mathrm{Cor}"] & \times & \mathrm{Sel}_{\mathfrak{k}_1}(A^t)_p  \arrow[rru, "F_{\mathfrak{k}_1}"] &  & 
\end{tikzcd}
\end{center}

\medskip

By this functoriality, we get a perfect pairing $$\displaystyle \lim_{\substack{\longleftarrow \\n}}(\mathrm{Sel}_{K_n}(A)_p)_{/div}  \times \lim_{\substack{\longrightarrow \\n}}(\mathrm{Sel}_{K_n}(A^t)_p)_{/div}  \rightarrow \mathbb{Q}_p/\mathbb{Z}_p$$ which is $\Lambda$-equivariant. Hence we have an isomorphism $$\displaystyle \lim_{\substack{\longleftarrow \\n}}(\mathrm{Sel}_{K_n}(A)_p)_{/div} \simeq  \lim_{\substack{\longleftarrow \\n}}\left(\mathrm{Sel}_{K_n}(A^t)_p^{\vee}[p^{\infty}]\right)$$ of $\Lambda$-modules.

\subsection{Proof of Theorem D}

\medskip

We first mention two technical lemmas without proof. This can be proved by using the explicit description of the functors $\mathfrak{F}$ and $\mathfrak{G}$. (See Proposition A.1.6 and Proposition A.2.12.)

\begin{lemma} (1) Let $M$ and $N$ be finitely generated torsion $\Lambda$-modules. If there are $\Lambda$-linear maps $\phi:M \rightarrow N$ and $\psi:N \rightarrow M$ with finite kernels, then $M$ and $N$ are pseudo-isomorphic.

\smallskip

(2) Let $X$ and $Y$ be finitely generated $\Lambda$-modules. Suppose that $\mathrm{rank}_{\mathbb{Z}_p}\frac{X}{\omega_{n}X}=\mathrm{rank}_{\mathbb{Z}_p}\frac{Y}{\omega_{n}Y}$ holds for all $n$, and that there are two $\Lambda$-linear maps $ \mathfrak{G}(X) \rightarrow \mathfrak{F}(Y), \mathfrak{G}(Y) \rightarrow \mathfrak{F}(X) $ with finite kernels. Then $E(X)$ and $E(Y)^{\iota}$ are isomorphic as $\Lambda$-modules. 
\end{lemma}

\medskip

Now we state the functional equation result between $\Lambda$-adic Selmer groups of $A$ and $A^t$. As we remarked in the introduction (Remark 1.2.12), the theorem below is a generalization of \cite[Theorem 1.14]{greenberg1999iwasawa} in two directions.

\medskip

\begin{theorem} If $\mathrm{Coker}(S_{n}^{A})$ and $\mathrm{Coker}(S_{n}^{A^t})$ are finite for all $n$, then we have an isomorphism $$E\left(\mathrm{Sel}_{K_{\infty}}(A)_p^{\vee}\right)\simeq E\left(\mathrm{Sel}_{K_{\infty}}(A^t)_p^{\vee}\right)^{\iota}$$
of $\Lambda$-modules. Here $\iota$ is an involution of $\Lambda$ satisfying $\iota(T)=\frac{1}{1+T}-1$.
\end{theorem}

\medskip

\begin{proof}[Proof of $\mathrm{Theorem}$ $5.3.3$] By Lemma 5.3.2, it suffices to show the following two assertions:

\begin{itemize}
\item $\mathrm{Sel}_{K_{\infty}}(A)_p[\omega_n]$ and $\mathrm{Sel}_{K_{\infty}}(A^t)_p[\omega_n]$ have same $\mathbb{Z}_p$-corank for all $n$.
\item There is a $\Lambda$-linear map $\mathfrak{G}\left(\mathrm{Sel}_{K_{\infty}}(A^t)_p^{\vee}\right) \rightarrow \mathfrak{F}\left(\mathrm{Sel}_{K_{\infty}}(A)_p^{\vee}\right)$ with the finite kernel.
\end{itemize}

The first assertion follows from our assumption about the finiteness of $\mathrm{Coker}(S_{n}^{A}), \mathrm{Coker}(S_{n}^{A^t})$ and Corollary 5.1.1. Now we prove the second statement by using Flach's pairing.\\

By the same method as the proof of Theorem 3.0.4, we have an exact sequence $$0 \rightarrow \mathrm{Coker}(S_n^{A^t})^{\vee}[p^{\infty}] \rightarrow \frac{\mathrm{Sel}_{K_{\infty}}(A^t)_p^{\vee}}{\omega_n \mathrm{Sel}_{K_{\infty}}(A^t)_p^{\vee}}[p^{\infty}] \rightarrow \mathrm{Sel}_{K_n}(A^t)_p^{\vee}[p^{\infty}] \rightarrow \mathrm{Ker}(S_n^{A^t})^{\vee}[p^{\infty}],$$  and isomorphisms $$\displaystyle \lim_{\substack{\longleftarrow \\ n}}\mathrm{Ker}(S_n^{A^t})^{\vee}[p^{\infty}]=\lim_{\substack{\longleftarrow \\ n}}\mathrm{Coker}(S_n^{A^t})^{\vee}[p^{\infty}]=0.$$
(Note that for the exact sequence, we used the finiteness of $\mathrm{Coker}(S_n^{A^t})$.) Now taking projective limit to the above sequence gives an isomorphism
\begin{align}
\displaystyle \mathfrak{G}\left(\mathrm{Sel}_{K_{\infty}}(A^t)_p^{\vee}\right) := \lim_{\substack {\longleftarrow \\ n}}\frac{\mathrm{Sel}_{K_{\infty}}(A^t)_p^{\vee}}{\omega_n \mathrm{Sel}_{K_{\infty}}(A)_p^{\vee}}[p^{\infty}] \simeq \lim_{\substack {\longleftarrow \\ n}}\mathrm{Sel}_{K_n}(A^t)_p^{\vee}[p^{\infty}].
\end{align}

\medskip

Now consider a natural exact sequence $$0 \rightarrow \mathrm{Ker}(S_n^{A}) \rightarrow \mathrm{Sel}_{K_n}(A)_p  \rightarrow \mathrm{Sel}_{K_{\infty}}(A)_p[\omega_n]  \rightarrow \mathrm{Coker}(S_n^{A}) \rightarrow 0.$$

By Lemma 2.1.4-(3) and the finiteness of the $\mathrm{Coker}(S_n^{A})$, we get another exact sequence $$\mathrm{Ker}(S_n^{A}) \rightarrow (\mathrm{Sel}_{K_n}(A)_p)_{/div} \rightarrow \left( \mathrm{Sel}_{K_{\infty}}(A)_p[\omega_n] \right) _{/div}  \rightarrow \mathrm{Coker}(S_n^{A}) \rightarrow 0  .$$ 

\smallskip

(Note that $\mathrm{Coker}(S_n^{A})$ is finite.) By taking projective limit, we get an exact sequence

\begin{align} \displaystyle 
\lim_{\substack{\longleftarrow \\n}}\mathrm{Ker}(S_n^{A}) \rightarrow \lim_{\substack{\longleftarrow \\n}}(\mathrm{Sel}_{K_n}(A)_p)_{/div} \rightarrow \lim_{\substack{\longleftarrow \\n}} (\mathrm{Sel}_{K_{\infty}}(A)_p[\omega_n])_{/div}. 
\end{align}

\medskip

\medskip

\medskip

We now analyze the three terms in this sequence (2).

\begin{itemize}
\item By the proof of Lemma  2.0.1, the first term $\displaystyle 
\lim_{\substack{\longleftarrow \\n}}\mathrm{Ker}(S_n^{A})$ injects into $\displaystyle 
\lim_{\substack{\longleftarrow \\n}} \frac{A(K_{\infty})[p^{\infty}]}{\omega
_n A(K_{\infty})[p^{\infty}]}$ which is a finite group. (This follows from the structure theorem of the $\Lambda$-modules and Lemma 2.0.1-(1).) 
\smallskip
\item The middle term in (2) is isomorphic to $\displaystyle \lim_{\substack {\longleftarrow \\ n}}\left(\mathrm{Sel}_{K_n}(A^t)_p^{\vee}[p^{\infty}]\right)$ by the remark mentioned before this subsection (The functorial property of the Flach's pairing), which is also isomorphic to $\mathfrak{G}\left(\mathrm{Sel}_{K_{\infty}}(A^t)_p^{\vee}\right)$ by (1).
\smallskip
\item Lastly, the third term in (2) is isomorphic to $\mathfrak{F}\left(\mathrm{Sel}_{K_{\infty}}(A)_p^{\vee}\right)$ by definition.
\end{itemize}

\medskip

Hence the sequence (2) becomes a $\Lambda$-linear map $\mathfrak{G}\left(\mathrm{Sel}_{K_{\infty}}(A^t)_p^{\vee}\right) \rightarrow \mathfrak{F}\left(\mathrm{Sel}_{K_{\infty}}(A)_p^{\vee}\right)$ with the finite kernel. 
\end{proof}

\medskip

\begin{prop} Suppose that $\mathrm{Coker}(S_{n}^{A})$, $\mathrm{Coker}(S_{n}^{A^t})$ are finite for all $n$. If $\sha_{K_n}^{1}(A)_p$ is finite for all $n$, then we have isomorphisms 
$$ E\left( \sha_{K_{\infty}}^{1}(A)_p^{\vee} \right) \simeq E\left( \sha_{K_{\infty}}^{1}(A^t)_p^{\vee} \right)^{\iota} $$
and
$$ E\left( (A(K_{\infty}) \otimes_{\mathbb{Z}_p} \mathbb{Q}_p/\mathbb{Z}_p)^{\vee} \right) \simeq E\left( (A^t(K_{\infty}) \otimes_{\mathbb{Z}_p} \mathbb{Q}_p/\mathbb{Z}_p)^{\vee} \right)^{\iota} $$
of $\Lambda$-modules.
\end{prop}

\medskip

We remark here that the finiteness of $\sha_{K_n}^{1}(A)_p$ implies the finiteness of $\sha_{K_n}^{1}(A^t)_p$ by \cite[Lemma I.7.1]{milne2006arithmetic}.

\medskip

\begin{proof} This follows from Corollary 3.0.6 and Theorem 5.3.3.
\end{proof}

\medskip

As a consequence of Theorem 5.3.3, we can compare the size of $\sha_{K_n}^{1}(A)_p$ and $\sha_{K_n}^{1}(A^t)_p$ over the tower of fields $ \lbrace K_n \rbrace_{n \geq 0}$.

\medskip

\begin{theorem} Suppose that $\mathrm{Coker}(S_{n}^{A})$, $\mathrm{Coker}(S_{n}^{A^t})$ are finite and bounded independent of $n$. If $\sha_{K_n}^{1}(A)_p$ are finite for all $n$, then the ratios $\frac{| \sha_{K_n}^{1}(A)_p |}{| \sha_{K_n}^{1}(A^t)_p |} $ and $\frac{| \sha_{K_n}^{1}(A^t)_p |}{| \sha_{K_n}^{1}(A)_p |} $ are bounded independent of $n$.
\end{theorem}

\medskip

\begin{proof} By Proposition 5.3.4, we have 
\begin{align*}
\mu\left(\sha_{K_{\infty}}^{1}(A)_p^{\vee}\right)&=\mu\left(\sha_{K_{\infty}}^{1}(A^t)_p^{\vee}\right) \\
\lambda\left(\sha_{K_{\infty}}^{1}(A)_p^{\vee}\right)&=\lambda\left(\sha_{K_{\infty}}^{1}(A^t)_p^{\vee}\right).
\end{align*}

Now applying the estimate of Theorem 4.0.1 to both groups $\sha_{K_n}^{1}(A)_p$ and  $\sha_{K_n}^{1}(A^t)_p $ gives the desired assertion.
\end{proof}

\appendix
\renewcommand{\thetheorem}{\Alph{section}.\arabic{subsection}.\arabic{theorem}}

\section{Functors $\mathfrak{F}$ and $\mathfrak{G}$}

\subsection{Functor $\mathfrak{F}$}

We first recall our convention on Pontryagin dual. For a locally compact Hausdorff continuous $\Lambda$-module $M$, we define $M^{\vee}:=\text{Hom}_{cts}(M, \mathbb{Q}_p/\mathbb{Z}_p)$ which is also a locally compact Hausdorff. $M^{\vee}$ becomes a continuous $\Lambda$-module via action defined by $(f\cdot \phi)(m):=\phi(\iota(f)\cdot m)$ where $f \in \Lambda$, $m \in M$, $\phi \in M^{\vee}$.

\medskip

If $M, N$ are two locally compact Hausdorff continuous $\Lambda$-modules with a perfect pairing $P:M \times N \rightarrow \mathbb{Q}_p/\mathbb{Z}_p$ satisfying $P(f \cdot x, y)=P(x, \iota(f) \cdot y)$, then $P$ induces $\Lambda$-module isomorphisms $M \simeq N^{\vee}$ and $N \simeq M^{\vee}$.

\medskip

\begin{definition}\label{definition A.2.4} For any finitely generated $\Lambda$-module $X$, define $\displaystyle \mathfrak{F}(X)=\left(\lim_{\substack {\longrightarrow  \\ n}} \frac{X}{\omega_{n}X} [p^{\infty}]\right)^{\vee}$. Here direct limit is taken with respect to norm maps $\frac{X}{\omega_{n}X} \xrightarrow{\times\frac{\omega_{n+1}}{\omega_{n}}} \frac{X}{\omega_{n+1}X}$. $\mathfrak{F}$ is contravariant and preserves finite direct sums.
\end{definition}

Our goal is examining $\Lambda$-module structure of $\mathfrak{F}(X)$ for a finitely generated $\Lambda$-module $X$ (Proposition A.1.6). We will use the following proposition crucially.

\begin{lemma}
\label{lemma A.2.5}
(1) Let $M$ be a finitely generated $\Lambda$-module, and let $\lbrace \pi_{n}\rbrace$ be a sequence of non-zero elements of $\Lambda$ such that $\pi_0 \in m, \pi_{n+1} \in \pi_{n}m,  \frac{M}{\pi_{n}M}$ is finite for all $n$ where $m$ is the maximal ideal of $\Lambda$. Then we have an isomorphism $$\displaystyle \left(\lim_{\substack{\longrightarrow \\ n}} \frac{M}{\pi_{n} M}\right)^{\vee}   \simeq \mathrm{Ext}^1_{\Lambda}(M, \Lambda)^{\iota}$$ as $\Lambda$-modules.

\medskip

(2) For a finitely generated $\Lambda$-module $X$, we have an isomorphism $\displaystyle \lim_{\substack{\longleftarrow \\ n}}T_{p}(X^{\vee}[\omega_n]) \simeq \mathrm{Hom}_{\Lambda}(X, \Lambda)^{\iota}$.
\end{lemma}

\begin{proof} We only need to prove (2) since (1) is \cite[Proposition 5.5.6]{neukirch2000cohomology}. By (1), we have an isomorphism $T_{p}(X^{\vee}[\omega_n]) \simeq \mathrm{Ext}^1_{\Lambda}(\frac{X}{\omega_n X}, \Lambda)^{\iota}$ and this last group is isomorphic to $$\text{Hom}_{\Lambda}(\frac{X}{\omega_n X}, \frac{\Lambda}{\omega_n \Lambda})^{\iota} \simeq \text{Hom}_{\Lambda}(X, \frac{\Lambda}{\omega_n \Lambda})^{\iota}.$$ Hence $\displaystyle \lim_{\substack{\longleftarrow \\ n}}T_{p}(X^{\vee}[\omega_n]) \simeq \displaystyle \lim_{\substack{\longleftarrow \\ n}}\text{Hom}_{\Lambda}(X, \frac{\Lambda}{\omega_n \Lambda})^{\iota} \simeq \mathrm{Hom}_{\Lambda}(X, \Lambda)^{\iota}$.
\end{proof}

\begin{lemma} \label{lemma A.2.6}
(1) If $K$ is finite, then $ \mathfrak{F}(K)=0 $

(2) $\mathfrak{F}(\Lambda)=0$

(3) If $g$ is a prime element of $\Lambda$ coprime to $\omega_{n}$ for all $n \geq 0$, then $\mathfrak{F}(\frac{\Lambda}{g^{e}})=\frac{\Lambda}{\iota(g)^e}$ for all $e \geq 1$

(4) \begin{equation*}
\mathfrak{F}(\frac{\Lambda}{\omega_{m+1, m} ^{e}})=\left\{
\begin{array}{cc}
\frac{\Lambda}{\iota(\omega_{m+1, m})^{e-1}} & e \geq 2,\\[5pt]
0 & e=1.
\end{array}
\right.
\end{equation*}
\end{lemma}

\smallskip

\begin{proof}
For (2), $\frac{\Lambda}{\omega_{n}\Lambda}$ is $\mathbb{Z}_p$-torsion-free so $\mathfrak{F}(\Lambda)=0$. For (3), note that $\frac{\Lambda}{(g^{e}, \omega_{n})}$ is finite if $g$ is a prime element of $\Lambda$ coprime to $\omega_{n}$ for all $n \geq 0$. So by Lemma \ref{lemma A.2.5}-(1), we get $$\mathfrak{F}\left(\frac{\Lambda}{g^{e}}\right) \simeq \text{Ext}_{\Lambda}^{1}(\frac{\Lambda}{g^{e}}, \Lambda)^{\iota}   \simeq \frac{\Lambda}{\iota(g)^{e}}.$$

For (4), we only prove for $e \geq 2$. Let $h=\omega_{m+1, m}$ and consider the following exact sequence for $n \geq m+1$:
$$0 \rightarrow \frac{\Lambda}{(h^{e-1}, \frac{\omega_{n}}{h})}\xrightarrow{\times h} \frac{\Lambda}{(h^{e}, \omega_{n})} \rightarrow \frac{\Lambda}{h} \rightarrow 0.$$ Since $\frac{\Lambda}{h}$ is $\mathbb{Z}_p$ torsion-free, we get $\frac{\Lambda}{(h^{e-1}, \frac{\omega_{n}}{h})}[p^{\infty}] \simeq \frac{\Lambda}{(h^{e}, \omega_{n})}[p^{\infty}]$. So we get \begin{align*}
\displaystyle \mathfrak{F}(\frac{\Lambda}{h^e})&=\left(\lim_{\substack{\longrightarrow \\ n}}\frac{\Lambda}{(h^{e}, \omega_{n})} [p^{\infty}] \right)^{\vee}=\left( \lim_{\substack{\longrightarrow \\ n}}\frac{\Lambda}{(h^{e-1}, \frac{\omega_{n}}{h})} [p^{\infty}] \right)^{\vee}\\
&=\left( \lim_{\substack{\longrightarrow \\ n}} \frac{\Lambda}{(h^{e-1}, \frac{\omega_{n}}{h})}\right)^{\vee} \quad \quad \text{ (Since $h$ is coprime to $\frac{\omega_{n}}{h}$ for all $n \geq m+1$ ) }\\
&\simeq \text{Ext}^{1}_{\Lambda}(\frac{\Lambda}{h^{e-1}}, \Lambda)^{\iota} \quad \quad \text{ (By Lemma \ref{lemma A.2.5}-(1)) }\\
&\simeq \frac{\Lambda}{\iota(h)^{e-1}}
\end{align*}
which shows the assertion.
\end{proof}

\begin{lemma}\label{lemma A.2.7} Let $0 \rightarrow K \rightarrow X \xrightarrow{\phi}  Z \rightarrow 0$ be a short exact sequence of finitely generated $\Lambda$ modules where $K$ is finite. Then the natural map $\mathfrak{F}(\phi) : \mathfrak{F}(Z) \rightarrow \mathfrak{F}(X)$ is an isomorphism.
\end{lemma}

\begin{proof}
Since the tensor functor is right exact, we have $\frac{K}{\omega_{n}K} \rightarrow \frac{X}{\omega_{n}X} \rightarrow \frac{Z}{\omega_{n}Z} \rightarrow 0.$

Let $A_{n}, B_{n}$ be the kernel and image of $\frac{K}{\omega_{n}K} \rightarrow \frac{X}{\omega_{n}X}$, recpectively. Note that both are finite modules since $K$ is finite. 
We have the following two short exact sequences
\begin{align}
0 \rightarrow A_{n} \rightarrow \frac{K}{\omega_{n}K} \rightarrow B_{n} \rightarrow 0   \\
0 \rightarrow B_{n} \rightarrow \frac{X}{\omega_{n}X} \rightarrow \frac{Z}{\omega_{n}Z} \rightarrow 0
\end{align}

By Lemma \ref{lemma A.1.3}-(2), two sequences remain exact after taking $p^{\infty}$-torsion parts. Since $\mathfrak{F}(K)=0$ by Lemma A.1.3, we get $\displaystyle (\lim_{\substack{\longrightarrow \\ n}} B_{n}[p^{\infty}])^{\vee} =0 $ from (3). Applying this to (4) gives an isomorphism $$\mathfrak{F}(\phi):\mathfrak{F}(Z)\xrightarrow{\mathfrak{F}(\phi)} \mathfrak{F}(X).$$
\end{proof}

\begin{lemma}\label{lemma A.2.8}
Let $0 \rightarrow Z \xrightarrow{\psi} S \rightarrow C \rightarrow 0 $ be a short exact sequence of finitely generated $\Lambda$ modules where $C$ is finite. Then the natural map $\mathfrak{F}(\psi):\mathfrak{F}(S) \rightarrow \mathfrak{F}(Z)$ is an injection with finite cokernel.

\end{lemma}

\begin{proof}
By the snake lemma, we have
$C[\omega_{n}] \rightarrow \frac{Z}{\omega_{n}Z} \rightarrow \frac{S}{\omega_{n}S} \rightarrow \frac{C}{\omega_{n}C} \rightarrow 0$.

\medskip

If we let $X_{n}=\text{Ker}(C[\omega_{n}] \rightarrow \frac{Z}{\omega_{n}Z}), E_{n}=\text{Im}(C[\omega_{n}] \rightarrow \frac{Z}{\omega_{n}Z}), D_{n}=\text{Ker}(\frac{S}{\omega_{n}S} \rightarrow \frac{C}{\omega_{n}C})$, then we get the following three short exact sequences:
\begin{align}0 \rightarrow X_{n} \rightarrow C[\omega_{n}] \rightarrow E_{n} \rightarrow 0 \\
0 \rightarrow E_{n} \rightarrow \frac{Z}{\omega_{n}Z} \rightarrow D_{n} \rightarrow 0 \\
0 \rightarrow D_{n} \rightarrow \frac{S}{\omega_{n}S} \rightarrow \frac{C}{\omega_{n}C} \rightarrow 0
\end{align}

Since $C$ is finite, $X_{n}$ and $E_{n}$ are finite for all $n$. So the sequences $(5), (6)$ remain exact after taking $p^{\infty}$-torsion parts due to Lemma \ref{lemma A.1.3}-(2). On the other hand, we have $\displaystyle\lim_{\substack{\longrightarrow \\ n}} \frac{C}{\omega_{n}C} =0$ and $\displaystyle\lim_{\substack{\longrightarrow \\ n}} C[\omega_{n}]=C$. So by taking direct limit and Pontryagin dual for the sequences $(5), (6), (7)$, we get $$0 \rightarrow \mathfrak{F}(S) \rightarrow \mathfrak{F}(Z) \rightarrow C^{\vee}.$$
\end{proof}

\medskip

Now we can prove our main goal.

\medskip

\begin{prop}\label{prop A.2.9}
Let $X$ be a finitely generated $\Lambda$-module. If we let $$\displaystyle E(X)  \simeq \Lambda^{r} 
\oplus \left(\bigoplus_{i=1}^{d} \frac{\Lambda}{g_{i}^{l_{i}}}\right) 
\oplus \left(\bigoplus_{\substack{m=1 \\ e_{1} \cdots e_{f} \geq 2}}^{f} \frac{\Lambda}{\omega_{a_{m}+1, a_{m}}^{e_{m}}} \right)
\oplus \left(\bigoplus_{n=1}^{t} \frac{\Lambda}{\omega_{b_{n}+1, b_{n}}}\right) $$ where $r \geq 0$, $g_{1}, \cdots g_{d}$ are prime elements of $\Lambda$ which are coprime to $\omega_{n}$ for all $n$, $d \geq 0$, $l_{1}, \cdots, l_{d} \geq 1$, $f \geq 0$, $e_{1}, \cdots, e_{f} \geq 2$ and $t \geq 0$, then we have an injection $$\displaystyle \left(\bigoplus_{i=1}^{d} \frac{\Lambda}{\iota(g_{i})^{l_{i}}}\right)\oplus \left(\bigoplus_{\substack{m=1 \\ e_{1} \cdots e_{f} \geq 2}}^{f} \frac{\Lambda}{\iota(\omega_{a_{m}+1, a_{m}})^{e_{m}-1}} \right) \hookrightarrow \mathfrak{F}(X)$$ with finite cokernel. In particular, $F(X)$ is a finitely generated $\Lambda$-torsion module.
\end{prop}

\begin{proof}
By the structure theorem of the finitely generated $\Lambda$-modules, we have an exact sequence $$0 \rightarrow K \rightarrow X \rightarrow E(X) \rightarrow C \rightarrow 0$$ where $K, C$ are finite modules. Let $Z$ be the image of $X \longrightarrow E(X)$. Now Lemma \ref{lemma A.2.6}, Lemma \ref{lemma A.2.7}, Lemma \ref{lemma A.2.8} applied to two short exact sequences $0 \rightarrow K \rightarrow X \rightarrow Z \rightarrow 0$ and $0 \rightarrow Z \rightarrow E(X) \rightarrow C \rightarrow 0$ give the desired assertion. 
\end{proof}

\medskip

\begin{cor}\label{cor A.2.10} Let $X$ be a finitely generated $\Lambda$-module and let $X_{\Lambda-\mathrm{tor}}$ be the maximal $\Lambda$-torsion submodule of $X$. If characteristic ideal of $X_{\Lambda-\mathrm{tor}}$ is coprime to $\omega_n$ for all $n$, then there is a pseudo-isomorphism $\phi:X \rightarrow \Lambda^{\mathrm{rank}_{\Lambda} X} \oplus \mathfrak{F}(X)^{\iota}$. If we assume additionally that $X$ does not have any non-trivial finite submodules, then $\phi$ is an injection.
\end{cor}

\medskip

\subsection{Functor $\mathfrak{G}$}

Now we consider another functor $\mathfrak{G}$.

\begin{definition}\label{definition A.3.11} For a finitely generated $\Lambda$-module $X$, define $\displaystyle \mathfrak{G}(X)=\lim_{\substack{\longleftarrow \\ n}}\left(\frac{X}{\omega_{n}X}[p^{\infty}]\right)$. Note that $\mathfrak{G}$ is covariant and preserves finite direct sums.
\end{definition}

\medskip

\begin{lemma}\label{lemma A.3.12}

(1) If $\frac{X}{\omega_{n}X}$ is finite for all $n$, then $\mathfrak{G}(X) \simeq X$. In particular, $\mathfrak{G}(\frac{\Lambda}{g^{e}}) \simeq \frac{\Lambda}{g^{e}}$ if $g$ is coprime to $\omega_{n}$ for all $n$.

(2) $\mathfrak{G}(\Lambda)=0.$

(3) \begin{equation*}
\mathfrak{G}(\frac{\Lambda}{\omega_{m+1, m}^{e}})=\left\{
\begin{array}{cc}
\frac{\Lambda}{\omega_{m+1, m}^{e-1}} & e \geq 2,\\[5pt]
0 & e=1.
\end{array}
\right.
\end{equation*}
\end{lemma}

\begin{proof}
For (1), since $\frac{X}{\omega_{n}X}$ is finite for all $n$, we get $\displaystyle \mathfrak{G}(X)=\lim_{\substack {\longleftarrow \\ n}}\frac{X}{\omega_{n}X}$ which is isomorphic to $X$. Since $\frac{\Lambda}{\omega_{n}\Lambda}$ is $\mathbb{Z}_p$-torsion-free, $\mathfrak{G}(\Lambda)=0$. Proof of (3) is almost same as that of Lemma \ref{lemma A.2.6}-(4).
\end{proof}

\medskip

\begin{lemma}\label{lemma A.3.13}
Let $0 \rightarrow K \rightarrow X \rightarrow Z \rightarrow 0$ be a short exact sequence of finitely generated $\Lambda$-modules where $K$ is finite. Then we have a short exact sequence $0 \rightarrow K \simeq \mathfrak{G}(K) \rightarrow \mathfrak{G}(X) \rightarrow \mathfrak{G}(Z) \rightarrow 0 $.
\end{lemma}

\medskip

\begin{proof} By the snake lemma, we get $Z[\omega_{n}] \rightarrow \frac{K}{\omega_{n}K} \rightarrow \frac{X}{\omega_{n}X} \rightarrow \frac{Z}{\omega_{n}Z} \rightarrow 0$. Define $E_{n}, A_{n}, B_{n}$ as $$E_{n}=\text{Ker}(Z[\omega_{n}] \rightarrow \frac{K}{\omega_{n}K}), A_{n}=\text{Im}(Z[\omega_{n}] \rightarrow \frac{K}{\omega_{n}K}), B_{n}=\text{Ker}(\frac{X}{\omega_{n}X} \rightarrow \frac{Z}{\omega_{n}Z}).$$ Then we get the following three exact sequences:
\begin{align}
0 \rightarrow E_{n} \rightarrow Z[\omega_{n}] \rightarrow A_{n} \rightarrow 0  \\
0 \rightarrow A_{n} \rightarrow \frac{K}{\omega_{n}K} \rightarrow B_{n} \rightarrow 0 \\
0 \rightarrow B_{n} \rightarrow \frac{X}{\omega_{n}X} \rightarrow \frac{Z}{\omega_{n}Z} \rightarrow 0
\end{align}
Since all the terms in sequence $(8), (9), (10)$ are compact, the sequences $(8), (9), (10)$ remain exact after taking projective limit. 

We can easily show that $\displaystyle \lim_{\substack{\longleftarrow \\n}}A_{n}=0$ and $\displaystyle \lim_{\substack{\longleftarrow \\n}}\frac{K}{\omega_{n}K} \simeq \lim_{\substack{\longleftarrow \\n}} B_{n}$. Since $B_{n}$ is finite, ($\because K$ is finite), the sequence $(10)$ remains exact after taking $p^{\infty}$-torsion parts. By taking projective limit, we get $$\displaystyle 0 \rightarrow \lim_{\substack{\longleftarrow \\n}}\frac{K}{\omega_{n}K} \simeq \lim_{\substack{\longleftarrow \\n}} B_{n} \rightarrow \lim_{\substack{\longleftarrow \\n}}\frac{X}{\omega_{n}X}[p^{\infty}]  \rightarrow \lim_{\substack{\longleftarrow \\n}}\frac{Z}{\omega_{n}Z}[p^{\infty}] \rightarrow 0$$ which gives the assertion.
\end{proof}

\medskip

\begin{lemma}\label{lemma A.3.14}
Let $0 \rightarrow Z \rightarrow S \rightarrow C \rightarrow 0 $ be a short exact sequence of finitely generated $\Lambda$ modules where $C$ is finite. Then we have an exact sequence $0 \rightarrow \mathfrak{G}(Z) \rightarrow \mathfrak{G}(S) \rightarrow \mathfrak{G}(C)\simeq C.$
\end{lemma}

\begin{proof} 
By the snake lemma, we get $C[\omega_{n}] \rightarrow \frac{M}{\omega_{n}M} \rightarrow \frac{S}{\omega_{n}S} \rightarrow \frac{C}{\omega_{n}C} \rightarrow 0$. Let $$A_{n}=\text{Ker}(C[\omega_{n}] \rightarrow \frac{M}{\omega_{n}M}), B_{n}=\text{Im}(C[\omega_{n}] \rightarrow \frac{M}{\omega_{n}M}), E_{n}=\text{Ker}(\frac{S}{\omega_{n}S} \rightarrow \frac{C}{\omega_{n}C}).$$ Then we get the following three short exact sequences:

\begin{align}
0 \rightarrow A_{n} \rightarrow C[\omega_{n}] \rightarrow B_{n} \rightarrow 0 \\
0 \rightarrow B_{n} \rightarrow \frac{M}{\omega_{n}M} \rightarrow E_{n} \rightarrow 0 \\
0 \rightarrow E_{n} \rightarrow \frac{S}{\omega_{n}S} \rightarrow \frac{C}{\omega_{n}C} \rightarrow 0
\end{align}

We can easily show that $\displaystyle \lim_{\substack{\longleftarrow \\n}}B_{n}=0$. Taking $p^{\infty}$-torsion parts from $(12), (13)$ gives the following two short exact sequences:
\begin{align}
0 \rightarrow B_{n} \rightarrow \frac{M}{\omega_{n}M}[p^{\infty}] \rightarrow E_{n}[p^{\infty}] \rightarrow 0 \\
0 \rightarrow E_{n}[p^{\infty}] \rightarrow \frac{S}{\omega_{n}S}[p^{\infty}] \rightarrow \frac{C}{\omega_{n}C}[p^{\infty}]  
\end{align}
Here, $(14)$ is exact due to Lemma \ref{lemma A.1.3}. \\

Since all the terms in $(14), (15)$ are finite, taking projective limit preserves  those two 
short exact sequences. Combining with $\displaystyle \lim_{\substack{\longleftarrow \\n}}B_{n}=0$ gives the assertion.
\end{proof}

Now we can prove our main goal for $\mathfrak{G}$, which is an analogue of Proposition \ref{prop A.2.9}.

\medskip

\begin{prop}\label{prop A.3.15}
Let $X$ be a finitely generated $\Lambda$-module. If we let $$\displaystyle E(X) \simeq \Lambda^{r} 
\oplus \left(\bigoplus_{i=1}^{d} \frac{\Lambda}{g_{i}^{l_{i}}}\right) 
\oplus \left(\bigoplus_{\substack{m=1 \\ e_{1} \cdots e_{f} \geq 2}}^{f} \frac{\Lambda}{\omega_{a_{m}+1, a_{m}}^{e_{m}}} \right)
\oplus \left(\bigoplus_{n=1}^{t} \frac{\Lambda}{\omega_{b_{n}+1, b_{n}}}\right) 
$$ where $r \geq 0$, $g_{1}, \cdots g_{d}$ are prime elements of $\Lambda$ which are coprime to $\omega_{n}$ for all $n$, $d \geq 0$, $l_{1}, \cdots, l_{d} \geq 1$, $f \geq 0$, $e_{1}, \cdots, e_{f} \geq 2$ and $t \geq 0$, then we have a pseudo-isomorphism $$\displaystyle \mathfrak{G}(X) \xrightarrow{\mathfrak{G}(\phi)} \left(\bigoplus_{i=1}^{d} \frac{\Lambda}{g_{i}^{l_{i}}}\right)\oplus \left(\bigoplus_{\substack{m=1 \\ e_{1} \cdots e_{f} \geq 2}}^{f} \frac{\Lambda}{\omega_{a_{m}+1, a_{m}}^{e_{m}-1}} \right).$$ In particular, $\mathfrak{G}(X)$ is a finitely generated $\Lambda$-torsion module.
\end{prop}

\medskip

\begin{proof} 
By the structure theorem of the finitely generated $\Lambda$-modules, we have an exact sequence $$0 \rightarrow K \rightarrow X \rightarrow E(X) \rightarrow C \rightarrow 0$$ where $K, C$ are finite modules. Let $Z$ be the image of $X \longrightarrow E(X)$. Now Lemma \ref{lemma A.3.12}, Lemma \ref{lemma A.3.13}, Lemma \ref{lemma A.3.14} applied to two short exact sequences $0 \rightarrow K \rightarrow X \rightarrow Z \rightarrow 0$ and $0 \rightarrow Z \rightarrow E(X) \rightarrow C \rightarrow 0$ give the desired assertion.
\end{proof}

Combining Proposition \ref{prop A.2.9} and Proposition \ref{prop A.3.15}, we get the following corollary.\\

\begin{cor}\label{cor A.3.16} For any finitely generated $\Lambda$-module $X$, two $\Lambda$-torsion modules $\mathfrak{F}(X)^{\iota}$ and $\mathfrak{G}(X)$ are pseudo-isomorphic.
\end{cor}

\bibliography{reference}
\bibliographystyle{alpha}

\end{document}